\documentclass{amsart}
\usepackage{latexsym}
\usepackage{enumerate}
\usepackage{amsmath, amssymb, amsfonts, amsxtra, amsthm}
\usepackage[active]{srcltx}

\newcommand {\rea}{\mathbb{R}}
\newcommand {\BBZ}{\mathbb{Z}}
\def\P{{\bf P}}
\def\T{{\bf T}}
\def\<{\left\langle}
\def\>{\right\rangle}
\newcommand {\size}{\mathop{\mathrm{size}}}
\newcommand {\sgn}{\mathop{\mathrm{sgn}}}
\newcommand {\ov}{\overline}
\newcommand {\calI}{\mathcal{I}}
\newcommand {\mfg}{\mathfrak{g}}
\newcommand {\mfb}{\mathfrak{b}}
\newcommand {\calD}{\mathcal{D}}
\newcommand {\calS}{\mathcal{S}}
\newcommand {\calN}{\mathcal{N}}
\newcommand {\diam}{\mathop{\mathrm{diameter}}}
\newcommand {\fraS}{\mathfrak{S}}
\newcommand {\fraA}{\mathfrak{A}}
\newcommand {\calA}{\mathcal{A}}

\newtheorem{theorem}{Theorem}[section]

\newtheorem{proposition}[theorem]{Proposition}

\newtheorem{lemma}[theorem]{Lemma}
\newtheorem{corollary}[theorem]{Corollary}

\author{Richard Oberlin}
\address{
Mathematics Department\\
Louisiana State University\\
Baton Rouge, LA}
\email{oberlin@math.lsu.edu}
\thanks{The author is supported in part by NSF Grant DMS-1068523.}


\begin{document}
\title[The Walsh model for $M^{q,*}$ Carleson]{Bounds on the Walsh model for $M^{q,*}$ Carleson and related operators}

\begin{abstract}
We prove an extension of the Walsh-analog of the Carleson-Hunt theorem, where the $L^\infty$ norm defining the Carleson maximal operator has been replaced by an $L^q$ maximal-multiplier-norm. Additionally, we consider certain associated variation-norm estimates.
\end{abstract}

\maketitle

\section{Introduction}

Given a real-valued function $f$ on $\rea^+$ consider the partial Walsh-sum operator, defined for $\xi,x \in \rea^{+}$
\begin{equation} \label{psdef}
\calS[f](\xi,x) = (\hat{f}\ 1_{[0,\xi)})\check{\ }(x).
\end{equation}
where $\hat{\ }$ and $\check{\ }$ refer to the Walsh-Fourier transform (terminology and notation will be explained in detail in Section \ref{defsection} and in a paragraph at the end of this section). The operator above can also be written using a wave-packet decomposition 
\[
\calS[f](\xi,x) = \sum_{P} \<f,\phi_{P_l}\> \phi_{P_l}(x)1_{\omega_{P_u}}(\xi)
\]
where we sum over all bitiles $P$ and $\phi_{P_l}$ is the $L^2$ normalized wave-packet corresponding to the lower half of $P$. Additionally, we will need a truncated version of $\calS$, defined for each integer $k$
\[
\calS_k[f](\xi,x) = \sum_{P : |I_P| < 2^k } \<f,\phi_{P_l}\> \phi_{P_l}(x)1_{\omega_{P_u}}(\xi).
\]

It is well-known that the Walsh-analog of the Carleson-Hunt theorem holds, namely that for $1 < p < \infty$
\begin{equation} \label{wchtheorem}
\|\calS[f](\xi,x)\|_{L^p_x(L^\infty_\xi)} \leq C_p \|f\|_{L^p}.
\end{equation}
We are interested in versions of the bound above where $L^\infty$ is replaced by various stronger norms. 

Given a function $m$ on $\rea^+$ and an exponent $1 \leq q \leq \infty$ consider the Walsh $L^q$-multiplier norm of $m$
\[
\|m\|_{M^q}  = \sup_{g : \|g\|_{L^q} = 1} \|(m \hat{g})\check{\ }\|_{L^q}.
\]
Replacing $m$ by a sequence of functions $\{m_k\}_{k \in \BBZ}$, one can also define a Walsh $L^q$-maximal-multiplier norm
\[
\|m\|_{M^{q,*}}  = \sup_{g : \|g\|_{L^q} = 1} \|(m_k \hat{g})\check{\ }(x)\|_{L^q_x(\ell^\infty_k)}.
\]
The Walsh-$M^{2,*}$-Carleson theorem was proven by Demeter, Lacey, Tao, and Thiele \cite{demeter08twm}. Specifically, they showed that if $1 < p < \infty$ then
\begin{equation} \label{walshM2starbound}
\|\calS_k[f](\xi,x)\|_{L^p_x(M^{2,*}_{\xi,k})} \leq C_p \|f\|_{L^p}.
\end{equation}
The main result of this paper is to extend the theorem above to cover exponents $q < 2,$\footnote{It seems likely that the range $q > 2$ is also tractable, see comments in \cite{demeter09osm}, however that case is not of particular interest for applications related to the return times theorem due to the monotonicity of $L^q$ norms in probability spaces.} namely we will prove 
\begin{theorem} \label{walshMpstar}
Suppose that $1 < p < \infty$ and $1 < q < 2$ satisfy $\frac{1}{p} + \frac{1}{q} < \frac{3}{2}.$ Then
\[
\|\calS_k[f](\xi,x)\|_{L^p_x(M^{q,*}_{\xi,k})} \leq C_{p,q} \|f\|_{L^p}.
\]
\end{theorem}

Since the $M^2$ norm is equal to the $L^\infty$ norm, one sees that the difference between \eqref{wchtheorem} and \eqref{walshM2starbound} is that the $M^2$ norm of the former bound is replaced by an $M^{2,*}$ norm in the latter bound, where the $^*$ refers to a maximum over truncations. Thus, when approaching Theorem \ref{walshMpstar}, one might first ask whether the corresponding bound holds with the $M^q$ norm in place of the $M^{q,*}$ norm. As we will now see, the affirmative answer to this question follows from combining work in \cite{oberlin09vnw}, which preceded a result for the Fourier-transform \cite{oberlin09vnc}, with the Walsh-analog of \cite{coifman88mdf}.

Given an exponent $r$ and a function $m$ defined on a subset of $\rea$ and taking values in some normed-linear space (in this paper, the subset of $\rea$ will be $\rea^+$ or $\BBZ$, and except for part of Section \ref{varmulsection} the normed linear space will be $\rea$) consider the $r$-variation norm
\[
\|m\|_{V^r} = \|m\|_{L^{\infty}} + \sup_{N,\xi_0 < \cdots < \xi_N} \left(\sum_{i=1}^{N} |m(\xi_{i}) - m(\xi_{i-1})|^r\right)^{1/r}.
\]
where the supremum is over all strictly increasing finite-length sequences in the domain of $m$.
It was proven in \cite{oberlin09vnw} (and we will give another proof here, see Section \ref{poftsection}) that if $r > 2$ and $p > r'$ then 
\begin{equation} \label{mainthm}
\|\calS[f](\xi,x)\|_{L^p_x(V^r_\xi)} \leq C_{p,r} \|f\|_{L^p}.
\end{equation}
Applying the method of Coifman, Rubio de Francia, and Semmes \cite{coifman88mdf} to Walsh-multipliers, one sees (as in Lemma \ref{Walshcrs} below) that if $r \geq 2$ and $|\frac{1}{q} - \frac{1}{2}| < \frac{1}{r}$ then for functions $m$ on $\rea^+$
\begin{equation} \label{crsWalshbound}
\|m\|_{M^q} \leq C_{q,r} \|m\|_{V^r}.
\end{equation}
Hence, it follows from \eqref{mainthm} that when $1 < p < \infty$ and $1 < q < \infty$ satisfy $\frac{1}{p} + \frac{1}{q} < \frac{3}{2}$
\begin{equation} \label{MqCarlesonbound}
\|\calS[f](\xi,x)\|_{L^p_x(M^{q}_{\xi})} \leq C_{p,q} \|f\|_{L^p}.
\end{equation}

It is thus clear that, as in \cite{demeter08twm}, the task at hand is to replace the $M^q$ norm in \eqref{MqCarlesonbound} with the $M^{q,*}$ norm. 
Roughly speaking, in \cite{demeter08twm} this advance was obtained by incorporating the use of Lemma \ref{varbourgainlemma} below into a proof of \eqref{wchtheorem} (this statement slurs over many technical obstructions, in fact their method required the development of a substantially new proof of \eqref{wchtheorem}).
We will follow the same approach, but with some necessary refinements which we now detail. 

First, we replace \eqref{crsWalshbound} and the natural $L^q$ extension of Lemma \ref{varbourgainlemma} with a common extension of the two bounds which is more efficient than their separate applications. We develop a new proof of \eqref{mainthm} which (as in the proof of \eqref{wchtheorem} from \cite{demeter08twm}) gives pointwise control for the sum over bitiles in a stack of trees in terms of the restrictions of the sum to individual trees. Obtaining this pointwise control for variation, rather than $L^\infty$, norms requires a more careful decomposition into $l$-overlapping trees, in particular the use of a concept of ``l-convexity''. This decomposition allows us to obtain an explicit partitioning of $\rea^+$ into intervals, on which the sum over a stack of trees agrees with its restriction to an individual tree. Finally, to control the variation-norm for an individual $l$-overlapping tree we use phase-space projections, as in \cite{demeter09osm}.

\subsection{Motivation} A significant part of our interest in Theorem \ref{walshMpstar} is due to its role as a model case for the corresponding Fourier-transform problem. Let $\Psi$ be (say) a Schwartz function on $\rea$ and for $f$ defined on $\rea$, $\xi,x \in \rea,$ and $k \in \BBZ$ consider the truncated partial Fourier-sum operator
\[
\fraS_k[f](\xi,x) = p.v. \int f(x - t) e^{2 \pi i \xi t} \Psi(2^{-k}t)\frac{1}{t}\ dt.
\]    
It was proven in \cite{demeter08bdr} that for $1 < p < \infty$ 
\begin{equation} \label{fourierM2starbound}
\|\fraS_k[f](\xi,x)\|_{L^p_x(M^{2,*}_{\xi,k})} \leq C_p \|f\|_{L^p}
\end{equation}
and it would be desirable to extend this result, as we have now done for the Walsh-model, to cover exponents $q < 2$.
One reason for interest in bounds such as \eqref{fourierM2starbound} is their application to the return times problem for the truncated Hilbert transform. It was shown in \cite{demeter08bdr} that bounds similar to \eqref{fourierM2starbound} can be used to deduce that given a measure preserving system $(X,T)$ and a function $f \in L^p(X), p > 1$ one can obtain a set $X'$ of full measure in $X$ such that for every $x \in X'$, every second measure preserving system $(Y,U)$, and every $g \in L^q(Y), q \geq 2$ the sums
\begin{equation} \label{htrtt}
\sideset{}{'}\sum_{n = -N}^N \frac{1}{n}f(T^nx)g(U^ny)
\end{equation}
converge as $N \rightarrow \infty$ for almost every $y \in Y.$ An extension of the range of exponents in \eqref{fourierM2starbound} could be used to extend the range of exponents for the pointwise convergence result.

Theorems similar to the convergence result above were originally considered \cite{bourgain88rtd} in the context of averages\footnote{which, at least morally, is an easier setting: one manifestation of this is that the analog of the Carleson-Hunt theorem for $\fraA:=\fraA_0$ is trivial.}, where one is interested in sums
\begin{equation} \label{srtt}
\frac{1}{N}\sum_{n = 1}^N f(T^nx)g(U^ny).
\end{equation}
Here, the relevant analog of $\fraS_k$ is
\[
\fraA_k[f](\xi,x) = \int f(x - t) e^{2 \pi i \xi t} 2^{-k}\Psi(2^{-k}t)\ dt.
\]
In \cite{demeter08bdr} it was shown that \eqref{fourierM2starbound} holds with $\fraA_k$ in place of $\fraS_k$ and this bound has been extended \cite{demeter09irt}, \cite{nazarov10czd} to cover the range $\frac{1}{p} + \frac{1}{q} < \frac{3}{2}$. The Walsh-analog of $\fraA_k$ would be 
\[
\calA_k[f](\xi,x) = \sum_{p : |I_p| = 2^k } \<f,\phi_{p}\> \phi_{p}(x)1_{\omega_{p}}(\xi)
\]  
where above we sum over all tiles $p$. It seems likely that bounds for $\calA_k$ could be obtained using a Walsh-analog of the method \cite{demeter09irt}.

\subsection{Further results}

Using a method from \cite{nazarov10czd}, see Section \ref{varmulsection}, one can obtain a variation-norm version of Lemma \ref{varbourgainlemma}. Substituting this lemma into the proof of Theorem \ref{walshMpstar} gives the stronger 
\begin{theorem} \label{walshMpvar}
Suppose that $s > 2$, $1 < p < \infty.$ and $1 < q < 2$ satisfy $\frac{1}{p} + \frac{1}{q} < \frac{3}{2}.$ Then
\begin{equation} \label{walshMpvarbound}
\|\calS_k[f](\xi,x)\|_{L^p_x(M^{q,s}_{\xi,k})} \leq C_{p,q,s} \|f\|_{L^p}
\end{equation}
\end{theorem}
\noindent where, given an exponent $s$, we define the $s$-variation-multiplier-norm of a sequence of functions $\{m_k\}_{k \in \BBZ}$
\[
\|m\|_{M^{q,s}}  = \sup_{g : \|g\|_{L^q} = 1} \|(m_k \hat{g})\check{\ }(x)|\|_{L^q_x(V^s_k)}.
\]
One reason for interest in bounds such as \eqref{walshMpvarbound} is that the analogous bounds for $\fraS_k$ and $\fraA_k$ would yield quantitative information about the convergence in \eqref{htrtt} and \eqref{srtt}. 

Through Corollaries \ref{pwcorollary2} and \ref{pwcorollary3}, we obtain the following variants of \eqref{mainthm}
\begin{theorem} \label{maxtruncvarmodtheorem}
Suppose that $r > 2$ and $p > r'.$ Then 
\[
\|\calS_k[f](\xi,x)\|_{L^p_x(\ell^{\infty}_k(V^r_\xi))} \leq C_{p,r} \|f\|_{L^p}.
\]
\end{theorem}
\begin{theorem} \label{maxmodvartrunctheorem}
Suppose that $r > 2$ and $1 < p < \infty.$ Then
\[
\|\calS_k[f](\xi,x)\|_{L^p_x(L^{\infty}_\xi(V^r_k))} \leq C_{p,r} \|f\|_{L^p}.
\]
\end{theorem}

The Fourier analog 
\begin{equation} \label{fmtvm}
\|\fraS_k[f](\xi,x)\|_{L^p_x(\ell^{\infty}_k(V^r_\xi))} \leq C_{p,r} \|f\|_{L^p}.
\end{equation}
of Theorem \ref{maxtruncvarmodtheorem} can be deduced from the Fourier-analog of \eqref{mainthm} by treating $\fraS_k$ as a superpositioning of modulated versions of $\fraS.$ If \eqref{fmtvm} held for exponents $r < 2$ (it doesn't), then the method of \cite{oberlin11amm} would allow one to deduce bounds of the type \eqref{fourierM2starbound} without using maximal-multiplier estimates such as Lemma 3.1.

The analog of Theorem \ref{maxmodvartrunctheorem} for $\fraS_k$ is related to the Wiener-Wintner theorem for the Hilbert transform \cite{lacey08wwt}, and was obtained \cite{oberlin09vnc} in the restricted range of exponents $r > 2, p > r'$ using the superposition argument.

The superposition argument does not seem to immediately apply to the Walsh-operator $\calS_k$ due to the different method of truncation.

\subsection{Structure of the paper}We give background information on the Walsh-Fourier transform in Section \ref{defsection}. Machinery is developed in Sections \ref{multsection} through \ref{tssection}. The machinery is applied to finish the proofs in Section \ref{poftsection}. Additional refinements needed for Theorem \ref{walshMpvar} are given in Section \ref{varmulsection}.

\subsection{Notational conventions}
We use $|\cdot|$ to Lebesgue measure, cardinality, or an understood norm depending on context. Given a rectangle $P = I \times \omega$ we let $I_P$ denote $I$ and $\omega_P$ denote $\omega.$ The indicator function of a set $E$ is written $1_E.$ Dyadic intervals are half-open on the right, i.e. of the form $[n2^k, (n+1)2^k)$ for integers $n,k.$

\section{Terminology of Walsh-phase-plane analysis} \label{defsection}
Given a nonnegative real number $x$ and an integer $n$, let $d_n(x)$ denote the digit which sits in the $n+1$'th position to the left of the point in the binary expansion of $x$, i.e. 
\[
x = \sum_{n} d_n(x) 2^n
\]
(for points $x$ on the dyadic grid, we choose the expansion with $d_n(x) = 0$ for $-n$ sufficiently large). Define the bitwise addition operation $\oplus$
 \[
d_n(x \oplus y) = d_n(x) + d_n(y)\mod 2, \ \ -\infty < n < \infty
\]
and the corresponding the multiplication operation $\otimes$
\[
d_n(x \otimes y) = \sum_{m} d_{n-m}(x)d_m(y)\mod 2, \ \ -\infty < n < \infty.
\]
Note that Lebesgue measure on $\rea^+$ is invariant under $\oplus$-translation.

Consider the character
\[
e(x) = e^{i \pi  d_{-1}(x)}
\]
and given a function $f$ on $\rea^+$ and $\xi \in \rea^+$ define the Walsh-Fourier transform
\[
\hat{f}(\xi)  = \int_{\rea^+} e(\xi \otimes x) f(x)\ dx.
\]
It is straightforward to verify that $\hat{1}_{[0,1)} = 1_{[0,1)}.$ Thus, after checking the identities $\widehat{f(x \oplus \cdot)} = e(x \otimes \cdot) \hat{f},\ \ \widehat{e^{\xi \otimes \cdot} f} = \hat{f}(\xi \oplus \cdot),$\ \ and $\widehat{f(2^k \cdot)} = 2^{-k}\hat{f}(2^{-k} \cdot),$ one sees that (for linear combinations of characteristic functions of dyadic intervals and hence by density for general functions $f \in L^2$) $\hat{\ }$ is involutive; however, for metaphorical purposes we will sometimes use the notation $\check{\ }$ in place of $\hat{\ }$.  

Given a dyadic ``time-interval'' $I \subset \rea^+$ and a dyadic ``frequency interval'' $\omega \subset \rea^+$, we say that the rectangle $I \times \omega$ is a tile if $|I \times \omega| = 1$ and we say that it is a bitile if $|I \times \omega| = 2.$ 
Each bitile $P$ contains four tiles $P_u, P_l, P_s,$ and $P_d$ which are the upper, lower, left, and right halves respectively. 
We impose the following partial order on the set of tiles and on the set of bitiles: 
\[
I_1 \times \omega_1 \leq I_2 \times \omega_2 \Leftrightarrow \omega_2 \subset \omega_1 \text{\ and\ } I_1 \subset I_2.
\]
A set of bitiles $\P$ is convex if for all bitiles $P_1 \leq P_2 \leq P_3$ with $P_1,P_3 \in \P$ we also have $P_2 \in \P$.
Through the use of standard limiting arguments we can, and will, assume that all bitiles belong to a finite convex set $\P_0$; all constants will be independent of this set.

Associated to each tile is the $L^2$-normalized Walsh wave-packet
\begin{equation} \label{wavepacketdef}
\phi_{I \times \omega}(x) = |I|^{1/2} \check{1}_{\omega}(x \oplus \inf I).
\end{equation}
Since $\phi_{I \times \omega}$ is supported on $I$ and $\hat{\phi}_{I \times \omega}$ is supported on $\omega$, $\phi_p$ and $\phi_{p'}$ are orthogonal unless the tiles $p,p'$ have nonempty intersection.
Letting $*$ denote the convolution operation 
\[
f * g (x) = \int_{\rea^+} f(x \oplus y)g(y)\ dy
\]
and letting $D_k[f](x)$ denote the average of $f$ over the dyadic interval of length $2^k$ containing $x$, we have
\[
D_k[f](x) = f * 2^{-k}1_{[0,2^k)}(x) 
\]
and, more generally, that for any dyadic interval $\omega$ 
\[
(\hat{f} 1_{\omega})\check{\ }(x) = \<f,\phi_{I_x \times \omega}\>\phi_{I_x \times \omega}(x)
\]
where $I_x$ is the dyadic interval of length $|\omega|^{-1}$ containing $x$. 

For each bitile $P$ we have the following relations between the wave-packets for the enclosed tiles
\begin{equation} \label{upperrelation}
\phi_{P_u} = \frac{1}{\sqrt{2}}(\phi_{P_s} - \phi_{P_d})
\end{equation} 
\begin{equation} \label{lowerrelation}
\phi_{P_l} = \frac{1}{\sqrt{2}}(\phi_{P_s} + \phi_{P_d}).
\end{equation}
The relations above can be used to check that our definition \eqref{wavepacketdef} agrees with that in \cite{demeter08twm}.
If a subset $S$ of $\rea^+\times\rea^+$ can be written as the disjoint union of tiles $S = \bigcup_{p \in \bf{p}}p$ we define the phase-plane projection 
\[
\Pi_Sf = \sum_{p \in \bf{p}}\<f,\phi_p\>\phi_p.
\]
When $S$ is a bitile, it follows immediately from \eqref{upperrelation} and \eqref{lowerrelation} that the projection is independent of the cover used in the definition. This is also true for general sets $S$, as can be seen by repeatedly appealing to the special case of the bitile. 
If $S \subset S'$ and $\Pi_S,\Pi_{S'}$ are both defined then 
\begin{equation} \label{projsubsetidentity}
\Pi_S \Pi_{S'} = \Pi_{S'} \Pi_S = \Pi_S.
\end{equation}

\section{Some multiplier estimates} \label{multsection}

In this section we recall an extension of a maximal-multiplier lemma of Bourgain, we recall a multiplier bound of Coifman, Rubio de Francia, and Semmes, and we then prove an estimate which is a hybrid of the two results. 

\subsection{A maximal-multiplier lemma}

Suppose that for every dyadic interval $\omega$ we have a coefficient $a_{\omega} \in \rea$ 
Let $\Xi \subset \rea^+$ be a finite collection of frequencies, and for each integer $k$ consider the Walsh-multiplier  
\[
\calD_k(\xi) = \sum_{\substack{|\omega| = 2^k \\ \omega \cap \Xi \neq \emptyset}} a_{\omega} 1_\omega(\xi)
\]
where, above, we sum over dyadic intervals $\omega$. Building on work of Bourgain \cite{bourgain89pet}, the following estimate was proven in \cite{demeter08twm}

\begin{lemma} \label{varbourgainlemma}
Let $r > 2$, and $\Xi \subset \rea^+.$ Then 
\[
\|\calD_k\|_{M^{2,*}} \leq C_{r}(1 + \log |\Xi|)|\Xi|^{\frac{1}{2} - \frac{1}{r}} \sup_{\xi \in \Xi} \|\sum_{|\omega| = 2^k}  a_{\omega}1_{\omega}(\xi)\|_{V^r_k}.
\]
\end{lemma}

In \cite{demeter09osm} and \cite{nazarov10czd} the Fourier-multiplier version of the estimate above was extended to $L^q$ for $1 < q < 2.$ Following the Walsh-analog of the argument in \cite{nazarov10czd} (which is part of the proof of Lemma \ref{hybridintervallemma} below) one would obtain
\begin{lemma} \label{Lpmaxnocomplemma}
Let $r > 2, 1 < q < 2,$ and $\Xi \subset \rea^+.$ Then 
\[
\|\calD_k\|_{M^{q,*}} \leq C_{q,r} (1 + \log|\Xi|)|\Xi|^{\frac{1}{q} - \frac{1}{r}} \sup_{\xi \in \Xi} \|\sum_{|\omega| = 2^k } a_{\omega}1_{\omega}(\xi)\|_{V^r_k}.
\]
\end{lemma}

We will need to use an estimate which encompasses both Lemma \ref{Lpmaxnocomplemma} and a separate multiplier bound which we will now discuss.

\subsection{Multipliers of bounded $r$-variation}

It was shown in \cite{coifman88mdf}, see also \cite{tao01emt}, that if $m$ is a function of bounded $r$-variation then the associated Fourier-multiplier operator is bounded on $L^q$ when $|\frac{1}{q} - \frac{1}{2}| < \frac{1}{r}$ and $1 < q < \infty$. 
The Walsh-multiplier analog of this result can be proven by the same method, which we now outline.

The first step is to obtain an estimate for multipliers given by linear combinations of characteristic functions of intervals.

\begin{lemma} \label{intervalmultiplierlemma}
Let $\epsilon > 0$, let $\Upsilon$ be a finite collection of pairwise disjoint subintervals of $\rea^+$, and let $\{b_{\upsilon}\}_{\upsilon \in \Upsilon} \subset \rea$ be a collection of coefficients. Then for $1 < q < \infty$ 
\[
\|\sum_{\upsilon \in \Upsilon}b_{\upsilon}1_{\upsilon}\|_{M^q} \leq C_{q,\epsilon} |\Upsilon|^{|\frac{1}{q} - \frac{1}{2}| + \epsilon}\sup_{\upsilon \in \Upsilon} |b_{\upsilon}|.
\]
\end{lemma}

\begin{proof}
In \cite{coifman88mdf}, the Fourier-multiplier version of this lemma was proven through the use of the Rubio de Francia square function estimate. One could follow the same route here by proving a Walsh-analog of the square function estimate, or by instead using \eqref{mainthm} with $r$ close to $2$ as a substitute for the square-function estimate. We will instead appeal to an estimate below using a multiple frequency Calder\'{o}n-Zygmund decomposition (this bears some similarity to the approach in \cite{tao01emt}). By duality one may assume $1 < q < 2.$ Choosing $r > 2$ and $\epsilon' > 0$ sufficiently small we have $\frac{1}{q} - \frac{1}{r} + \epsilon' < \frac{1}{q} - \frac{1}{2} + \epsilon$. Taking $\Xi = \{0\}$ and $a_{\omega} = 1$ for every $\omega$ we then have 
\[
(\sum_{\upsilon \in \Upsilon}b_{\upsilon}1_{\upsilon}\hat{g})\check{\ }(x) = \lim_{k \rightarrow \infty} (\calD_k \sum_{\upsilon \in \Upsilon}b_{\upsilon}1_{\upsilon}\hat{g})\check{\ }(x)
\] 
almost everywhere and so the lemma follows immediately from Lemma \ref{hybridintervallemma}.

\end{proof}

Next, we see (in a lemma directly from \cite{coifman88mdf}) that functions of bounded $r$-variation can efficiently be written as sums of functions of the type treated in Lemma \ref{intervalmultiplierlemma}. 

\begin{lemma} \label{intervaldecompositionlemma}
Let $m$ be a compactly supported function on $\rea^+$ of bounded $r$-variation for some $1 \leq r < \infty.$ Then for each integer $j \geq 0$, one can find a collection $\Upsilon_j$ of pairwise disjoint subintervals of $\rea^+$ and coefficients $\{b_{\upsilon}\}_{\upsilon \in \Upsilon_j} \subset \rea$ so that $|\Upsilon_j| \leq 2^j$, $|b_{\upsilon}| \leq \|m\|_{V_r} 2^{-j/r}$, and
\[
m = \sum_{j \geq 0} \sum_{\upsilon \in \Upsilon_j} b_{\upsilon} 1_{\upsilon} 
\]
where the sum in $j$ converges uniformly.
\end{lemma} 

\begin{proof}
The proof is exactly as in \cite{coifman88mdf} (also see \cite{lacey07irt}) so we will be short with the details. Choose $B$ so that $m$ is supported on $[0,B].$ Set $V(0) = 0$ and for each $x \in (0,B]$ let
\[
V(x) = \sup_{N,0 = \xi_0 < \ldots < \xi_N = x} \sum_{k = 1}^N |m(\xi_{k}) - m(\xi_{k-1})|^r.
\]
For $j \geq 0$ and $1 \leq l < 2^j$ let 
\[
\upsilon_{j,l} = V^{-1}\left([(l-1) 2^{-j} \|m\|_{V^r}^r, l 2^{-j} \|m\|_{V^r}^r )\right)
\]
and 
\[
\upsilon_{j,2^j} = V^{-1}\left([(2^j-1) 2^{-j} \|m\|_{V^r}^r, \|m\|_{V^r}^r ]\right).
\]
Set $\tilde{b}_{\upsilon_{j,l}} = 0$ if $\upsilon_{j,l} = \emptyset$, $\tilde{b}_{\upsilon_{j,l}} = m(\upsilon_{j,l})$ if $\upsilon_{j,l}$ is a singleton set, and 
\[
\tilde{b}_{\upsilon_{j,l}} = \frac{1}{|\upsilon_{j,l}|} \int_{\upsilon_{j,l}} m(x)\ dx
\]
if $\upsilon_{j,l}$ is a non-singleton interval. 
Then letting $\Upsilon_{j} = \{\upsilon_{j,l}\}_{l=1}^{2^j}$, $b_{\upsilon_{0,l}} = \tilde{b}_{\upsilon_{0,l}}$, and $b_{\upsilon_{j,l}} = \tilde{b}_{\upsilon_{j,l}} - \tilde{b}_{\upsilon_{j - 1,l}}$ for $j > 0$, one sees that the requirements of the lemma are satisfied.
\end{proof}

Finally we combine
Lemmas \ref{intervalmultiplierlemma} and \ref{intervaldecompositionlemma} to obtain the Walsh-analog of a result from \cite{coifman88mdf}\footnote{Strictly speaking, in \cite{coifman88mdf} they considered the case where $\Upsilon$ was a collection of dyadic shells and through an additional Littlewood-Paley argument were able to obtain a norm bound which did not blow up with $|\Upsilon|.$ To match our application, we are more flexible with $\Upsilon$ and can accept the resulting loss in the bound.}
\begin{lemma} \label{Walshcrs}
Let $1 < q < \infty$, $|\frac{1}{q} - \frac{1}{2}| < \frac{1}{r}$, and $\epsilon > 0$. 
Suppose 
$\Upsilon$ is a collection of pairwise disjoint subintervals of $\rea^+$ and that for each $\upsilon \in \Upsilon$, $m_{\upsilon}$ is a function supported on $\upsilon$ with $ \|m\|_{V^r} < \infty$. Then
\[
\|\sum_{\upsilon \in \Upsilon}m_{\upsilon}\|_{M^q} \leq C_{q,r,\epsilon} |\Upsilon|^{|\frac{1}{q} - \frac{1}{2}| + \epsilon} \sup_{\upsilon \in \Upsilon}\|m_{\upsilon}\|_{V^r}.
\]
\end{lemma}

\begin{proof}

After a limiting argument, one may assume that all intervals in $\Upsilon$ have finite length. Applying Lemma \ref{intervaldecompositionlemma} to each $m_{\upsilon}$ we obtain for $j \geq 0$ a collection $\calI_{\upsilon,j}$ of at most $2^j$ pairwise disjoint subintervals of $\upsilon$ and coefficients $\{b_I\}_{I \in \calI_{\upsilon,j}}$ so that 
\[
m_{\upsilon} = \sum_{j \geq 0} \sum_{I \in \calI_{\upsilon,j}} b_{I} 1_I.
\]
Then
\[
\|\sum_{\upsilon \in \Upsilon} m_{\upsilon} \|_{M^q} \leq \sum_{j \geq 0} \|\sum_{\upsilon \in \Upsilon} \sum_{I \in \calI_{\upsilon,j}}b_I 1_I \|_{M^q}.
\]
Applying Lemma \ref{intervalmultiplierlemma} with the collection of pairwise disjoint intervals $\bigcup_{\upsilon \in \Upsilon} \calI_{\upsilon,j}$ we see that each term on the right above is 
\begin{align*}
&\leq C_{q,\epsilon}  (2^j|\Upsilon|)^{|\frac{1}{q} - \frac{1}{2}| + \epsilon} \sup_{\upsilon \in \Upsilon, I \in \calI_{\upsilon,j}} |b_{I}| \\
&\leq C_{q,\epsilon} (2^j|\Upsilon|)^{|\frac{1}{q} - \frac{1}{2}| + \epsilon} 2^{-\frac{j}{r}} \sup_{\upsilon \in \Upsilon} \|m_{\upsilon}\|_{V^r}. 
\end{align*}
The sum over $j \geq 0$ converges after possibly shrinking $\epsilon$ to satisfy $|\frac{1}{q} - \frac{1}{2}| + \epsilon < \frac{1}{r}.$
\end{proof}

\subsection{A hybrid estimate}

Our aim here is to prove the following lemma which, except for a restriction on the range of $r$ and a difference in the dependence on $|\Upsilon|$, is a common extension of Lemmas \ref{Lpmaxnocomplemma} and \ref{Walshcrs}.

\begin{lemma} \label{hybridlemma}
Let $1 < q \leq 2$, $2 < r < 2q$, $\epsilon > 0,$ and $\Xi \subset \rea^+.$ Suppose that $\Upsilon$ is a collection of pairwise disjoint subintervals of $\rea^+$ and that $\{m_{\upsilon}\}_{\upsilon \in \Upsilon}$ is a collection of functions of bounded $r$-variation such that each $m_{\upsilon}$ is supported on $\upsilon.$ Then
\[
\|\calD_k \sum_{\upsilon \in \Upsilon} m_{\upsilon}\|_{M^{q,*}} \leq C_{q,r,\epsilon}  (|\Xi| + |\Upsilon|)^{\frac{1}{q} - \frac{1}{r} + \epsilon} \sup_{\xi \in \Xi}  \|\sum_{|\omega| = 2^k } a_{\omega}1_{\omega}(\xi)\|_{V^r_k} \sup_{\upsilon \in \Upsilon} \|m_{\upsilon}\|_{V^r}.
\]
\end{lemma}

A version of the lemma above, but with $(|\Xi| + |\Upsilon|)^{\frac{1}{q} - \frac{1}{r} + \epsilon}$ replaced by 
\\ $|\Xi|^{\frac{1}{q} - \frac{1}{r} + \epsilon} |\Upsilon|^{\frac{1}{q} - \frac{1}{2} + \epsilon}$
would follow by combining Lemmas \ref{Lpmaxnocomplemma} and \ref{Walshcrs} and estimating the operator norm of the composition by the product of the operator norms. In our application we will take $r$ arbitarily close to $2$ and $|\Upsilon| = |\Xi|$; thus, the norm bound obtained above improves substantially on the combination of the two prior lemmas; this improvement seems to be necessary to obtain the desired range of exponents in Theorem \ref{walshMpstar}.

The new ingredient needed in the proof of Lemma \ref{hybridlemma}  is the following hybrid of Lemmas \ref{Lpmaxnocomplemma} and \ref{intervalmultiplierlemma}.

\begin{lemma} \label{hybridintervallemma}
Let $r > 2$, $1 < q \leq 2,$ and $\Xi \subset \rea^+.$ Suppose that $\Upsilon$ is a collection of pairwise disjoint subintervals of $\rea^+$ and that $\{b_{\upsilon}\}_{\upsilon \in \Upsilon} \subset \rea$ is a collection of coefficients. Then
\begin{multline*}
\|\calD_k \sum_{\upsilon \in \Upsilon} b_{\upsilon}1_{\upsilon}  \|_{M^{q,*}} \leq \\ C_{q,r}  (1 + \log(|\Xi| + |\Upsilon|))(|\Xi| + |\Upsilon|)^{\frac{1}{q} - \frac{1}{r}} \sup_{\xi \in \Xi}  \|\sum_{|\omega| = 2^k } a_{\omega}1_{\omega}(\xi)\|_{V^r_k} \sup_{\upsilon \in \Upsilon}|b_{\upsilon}|.
\end{multline*}
\end{lemma}

\begin{proof}
Through a limiting argument, one may assume that all intervals in $\Upsilon$ have finite length.

The desired bound at $q=2$ follows immediately from Lemma \ref{varbourgainlemma}, and so by interpolation it suffices to prove a weak-type estimate at $q=1$. Specifically, given $g \in L^1$ we need to show that for each $\lambda > 0$
\[
|\{x : \sup_k|(\calD_k\sum_{\upsilon \in \Upsilon} b_{\upsilon} 1_\upsilon \hat{g})\check{\ }(x)| > \lambda\}| \leq C N^{1/2} B \|g\|_{L^1}/\lambda
\]
where $N = |\Xi| + |\Upsilon|$ and 
\[
B =(1 + \log(N)) N^{\frac{1}{2} - \frac{1}{r}} \sup_{\xi \in \Xi}  \|\sum_{|\omega| = 2^k} a_{\omega}1_{\omega}(\xi)\|_{V^r_k} \sup_{\upsilon \in \Upsilon}|b_{\upsilon}|.
\]
  
We start by performing a multiple-frequency Calder\'{o}n-Zygmund decomposition. Let  
\[
E = \{x : M[g](x) > \lambda/(N^{1/2}B)\} 
\]
where $M$ is the dyadic version of the Hardy-Littlewood maximal operator. Let $\calI$ be the collection of maximal dyadic intervals contained in $E$ and 
\[
\Lambda = \Xi \cup \bigcup_{\upsilon \in \Upsilon} \{\inf \upsilon,\sup \upsilon\}.
\] 
We now construct the ``good function'' $\mfg$. Let 
\[
\mfg_0 = \sum_{I \in \calI} \sum_{\substack{|\omega| = |I|^{-1} \\ \omega \cap \Lambda \neq \emptyset}} \<g,\phi_{I \times \omega}\>\phi_{I \times \omega}
\]
where the second sum above is over dyadic intervals $\omega$. Setting 
\[
\mfg = \mfg_0 + 1_{\rea^+ \setminus E} g
\] 
we obtain the ``bad function''
\[
\mfb = g - \mfg.
\]
and write $\mfb_I$ in place of $1_I \mfb$ for each $I \in \calI.$

The contribution from the good function is controlled, as usual, by the previously known $L^2$ bound. Indeed, by the maximality of the intervals $I$, we have
\[
|\<g,\phi_{I \times \omega}\>| \leq 2 |I|^{1/2} \lambda/(N^{1/2}B).
\]
Using the orthogonality of the wavepackets and the fact that $|\Lambda| \leq 2 N$ we have
\[
\|1_I \mfg_0\|_{L^2} \leq C |I|^{1/2} \lambda/B. 
\]
This gives 
\[
\|\mfg_0\|_{L^2}^2 \leq C |E| \lambda^2/B^2 \leq C N^{1/2} \|g\|_{L^1} \lambda/B.
\]
Since $g$ is bounded by $\lambda/(N^{1/2}B)$ away from $E$, we have
\[
\|\mfg - \mfg_0\|_{L^2}^2 \leq N^{-1/2} \|g\|_{L^1} \lambda/B.
\]
Thus
\begin{align*}
|\{x : \sup_k |(\calD_k \sum_{\upsilon \in \Upsilon} b_{\upsilon} 1_\upsilon \hat{\mfg})\check{\ }(x)| > \lambda/2\}| &\leq 4 \|\sup_k|(\calD_k \sum_{\upsilon \in \Upsilon} b_{\upsilon} 1_\upsilon \hat{\mfg})\check{\ }|\|_{L^2}^2/\lambda^2 \\
& \leq C B^2 \|\mfg\|_{L^2}^2/\lambda^2 \\
& \leq C N^{1/2} B \|g\|_{L^1}/\lambda
\end{align*}
as desired. 

It remains to control the contribution from the bad function. The two important properties of $\mfb$ are that it is supported on $E$ and that for each $I$ in $\calI$ we have $\<\mfb,\phi_{I \times \omega}\> = 0$ for every dyadic interval $\omega$ with $|\omega| = |I|^{-1}$ and $\omega \cap \Lambda \neq \emptyset.$ We claim that the function
\[
h = (\sum_{\upsilon \in \Upsilon} b_{\upsilon} 1_\upsilon \hat{\mfb})\check{\ }
\]
shares these two properties with $\mfb$. We first consider the support property. Fix $I \in \calI,$ suppose $\upsilon \in \Upsilon$, and let $\omega$ be a maximal dyadic subinterval of $\upsilon.$ Then 
\[
(1_{\omega} \hat{\mfb}_I)\check{\ }(x) = \<\mfb_I,\phi_{J \times \omega}\>\phi_{J \times \omega}(x)
\]
where $J$ is the dyadic interval of length $|\omega|^{-1}$ containing $x$. We clearly have $\<\mfb_I,\phi_{J \times \omega}\> = 0$ if $J$ has empty intersection with $I$. If $x \notin I$ and $J$ intersects $I$ then we have $I \subsetneq J$ and in particular $|I| < |J|.$ By the maximality of $\omega$, the dyadic interval $\tilde{\omega}$ of length $|I|^{-1}$ containing $\omega$ intersects $\{\sup \upsilon, \inf \upsilon\}$ and hence $\<\mfb_I,\phi_{I \times \tilde{\omega}}\> = 0.$ Using the fact that the restriction of $\phi_{J \times \omega}$ to $I$ is a constant multiple of $\phi_{I \times \tilde{\omega}}$ we see that $\<\mfb_I, \phi_{J \times \omega}\> = 0.$ Since each $\upsilon$ can be written as the union of maximal dyadic subintervals, this implies that 
$(\sum_{\upsilon \in \Upsilon}b_{\upsilon} 1_{\upsilon} \hat{\mfb}_I)\check{\ }$ is supported on $I$ and so $h$ is supported on $E$.
To verify the cancellation property of $h$, we let $I \in \calI$ and let $\omega$ be a dyadic interval of length $|I|^{-1}$ such that $\omega \cap \Lambda \neq \emptyset.$ Then $\<h,\phi_{I \times \omega}\>$ is zero simultaneously with the restriction of $(1_\omega\widehat{1_I h})\check{}$ to $I$. However, $1_Ih = (\sum_{\upsilon \in \Upsilon}b_{\upsilon} 1_{\upsilon} \hat{\mfb}_I)\check{\ }$ and so $1_\omega\widehat{1_I h} = \sum_{\upsilon \in \Upsilon}b_{\upsilon} 1_{\upsilon} 1_{\omega}\hat{\mfb}_I.$ From the cancellation property of $\mfb$, we know that $1_{\omega}\hat{\mfb}_I$ is identically zero.

Arguing as in the previous paragraph, except with $h$ in place of $\mfb$, one sees that each $(\calD_k \hat{h})\check{\ }$ is supported on $E$ and thus 
\[
|\{\sup_k |(\calD_k h)\check{\ }(x)| > \lambda/2\}| \leq |E| \leq N^{1/2} B \|g\|_{L^1}/\lambda.
\]
\end{proof}

\begin{proof}[Proof of Lemma \ref{hybridlemma}]
Following the argument in Lemma \ref{Walshcrs}, but with Lemma \ref{hybridintervallemma} substituted for Lemma \ref{intervalmultiplierlemma} we see that 
\begin{align*}
\|\calD_k&\sum_{\upsilon \in \Upsilon} m_{\upsilon} \|_{M^{q,*}}\\
&\leq \sum_{j \geq 0}C_{q,r,\epsilon'}  (|\Xi| + 2^j|\Upsilon|)^{\frac{1}{q} - \frac{1}{r} + \epsilon'} \sup_{\xi \in \Xi}  \|\sum_{|\omega| = 2^k} a_{\omega}1_{\omega}(\xi)\|_{V^r_k} \sup_{\upsilon \in \Upsilon, I \in \calI_{\upsilon,j}} |b_{I}|  \\
&\leq \sum_{j \geq 0}C_{q,r,\epsilon'}  2^{j(\frac{1}{q} - \frac{1}{r} + \epsilon')}(|\Xi| + |\Upsilon|)^{\frac{1}{q} - \frac{1}{r} + \epsilon'} \sup_{\xi \in \Xi}  \|\sum_{|\omega| = 2^k} a_{\omega}1_{\omega}(\xi)\|_{V^r_k} 2^{-\frac{j}{r}} \sup_{\upsilon \in \Upsilon} \|m_{\upsilon}\|_{V^r}. 
\end{align*}
The sum over $j \geq 0$ then converges provided that $r < 2q$ and $\epsilon' < \epsilon$ is chosen sufficiently small.
\end{proof}

\section{Tree decompositions}

Given a collection of bitiles $T$, a ``top frequency'' $\xi_T \in \rea^+$, and a dyadic ``top interval'' $I_T \subset \rea^+$, we say that $(T,\xi_T,I_T)$ form a tree if $I_P \subset I_T$ and $\xi_T \in \omega_P$ for every $P \in T.$ We say that a tree $T$ is ``td-maximal'' among trees in a collection $\T$ if it is maximal with respect to inclusion among trees in $\T$ with top data $(\xi_T,I_T).$ 
Given any convex tree $T$, we can rewrite $\cup_{P \in T}P$ as a disjoint union of tiles. For such a tree, we abbreviate
$$\Pi_{T}=\Pi_{\bigcup P: P\in T}\ \ .$$

For a convex collection $\P$ of bitiles define
$$\size(\P,f)=\sup_T |I_T|^{-1/2} \|\Pi_{T} f\|_{L^2}.$$
where the sup is over all convex trees $T \subset \P.$
Note that (since the $L^{\infty}$ norm can be controlled by projections to individual subtiles of elements of $T$) for each convex tree $T$ and $1 \leq p \leq \infty$ we have
\[
\|\Pi_Tf\|_{L^p} \leq C \size(T,f) |I_T|^{1/p}.
\]

The following lemma was proven in \cite{oberlin10nub}:

\begin{lemma}[{\bf Tree Selection}]\label{treeselection}
Assume $\P$ is a finite convex collection of bitiles with $\size(\P,f)\le 2^{-k}$. Then we can write
$\P$ as the disjoint union of a convex set of bitiles $\P'$ with the union of a collection $\T$
of convex trees such that
\begin{equation}\label{l1N}
\|\sum_{T\in \T} 1_{I_T}\|_1 \le C 2^{2k}\|f\|_2^2\ \ ,
\end{equation}
\begin{equation}\label{bmoN}
\|\sum_{T\in \T} 1_{I_T}\|_{\rm BMO} \le C 2^{2k}\|f\|_\infty^2\ \ ,
\end{equation}
and 
$\size(\P',f)\le 2^{-k-1}$.
\end{lemma}

Strictly speaking, the lemma above was proven with a different definition of tree -- in \cite{oberlin10nub} a tree is a collection of bitiles with a unique maximal element, we will call this an m-tree. It is easily seen that the lemma for m-trees implies the lemma for trees since every m-tree is a tree and every finite convex tree $T$ contains a convex m-tree $T'$ with $\size(T',f) = \size(T,f).$

We also note that we may assume that the trees $T$ in the lemma above are td-maximal with among trees contained in $\bigcup_{T \in \T}T;$ this is accomplished by taking them to be td-maximal among trees contained in $\P.$ 

Finally, by following the proof of \eqref{bmoN} one sees that for any $\T' \subset \T$ we also have
\begin{equation*}%\label{bmoNsub}
\|\sum_{T\in \T'} 1_{I_T}\|_{\rm BMO} \le C 2^{2k}\|f\|_\infty^2
\end{equation*}

We say that a tree $T$ is $l$-overlapping if $\xi_T \in \omega_{P_l}$ for every $P \in T$ and $u$-overlapping if $\xi_T \in \omega_{P_u}$ for every $P \in T.$ We call a set of bitiles $\P$ ``$l$-convex'' if $P, P'' \in \P$ and $P_l < P_l' < P_l''$ imply that $P_l' \in \P$. 
Finally, a collection of $l$-overlapping trees $\T$ will be called
``properly-sorted'' if the following conditions hold for each $T \in \T$:

\newcounter{listcounter}
\begin{list}{(\arabic{listcounter})}{\usecounter{listcounter}\setcounter{listcounter}{\value{equation}}}
%\item \label{propsorta}For every $T' \in \T$ with $T' \neq T$ and $\xi_{T'} = \xi_{T}$ we have $I_{T'} \cap I_{T} = \emptyset.$
\item  \label{propsortb}$T$ is $l$-convex 
\item \label{propsorta}For every $T' \in \T \setminus \{T\}$ we have $T' \cap T = \emptyset$
\item  \label{propsortc}For every $P \in T$ and $T' \in \T$ with $I_{T'} \cap I_P \neq \emptyset$ we have $\xi_{T'} \notin \omega_{P_u}.$
\end{list}\setcounter{equation}{\value{listcounter}}

The importance of condition \eqref{propsortb} is that when an $l$-overlapping tree $T$ is $l$-convex, 
\begin{equation} \label{topintervalinterval}
\bigcup_{\substack{P \in T \\ x \in I_P}} \omega_{P_u}
\end{equation}
is an interval for each $x$. Indeed, suppose that $\xi_1 < \xi_2 < \xi_3$ with $\xi_1,\xi_3$ in the set \eqref{topintervalinterval}. Then, there are bitiles $P^i \in T$ with $\xi_i \in \omega_{P^i_u}$ and $x \in I_{P^i}$ for $i = 1,3$; clearly $P^3_l \leq P^1_l.$
Let $\omega$ be the smallest dyadic interval with $\xi_T,\xi^2 \in \omega,$ let $I$ be the dyadic interval of length $2|\omega|^{-1}$ containing $x$, and let $P^2$ be the bitile $I \times \omega.$ Then $P^3_l \leq P^2_l \leq P^1_l,$ so by $l$-convexity $P^2 \in T$ and hence $\xi_2$ is in the set \eqref{topintervalinterval}.

The condition \eqref{propsortc} is taken from \cite{demeter08twm}; one of its immediate consequences is that the tiles $P_u$ with $P \in \bigcup_{T \in \T}$ are disjoint. Indeed, $P_u < P'_u$ would imply that the tree containing $P'_u$ has top frequency contained in $\omega_{P_u}.$

We will apply the lemma below to the collection of bitiles $\P \setminus \P'$ from Lemma \ref{treeselection}. 

\begin{lemma} \label{propsortlemma}
Suppose that a finite collection of bitiles $\P$ can be written as the union of convex trees $\P = \bigcup_{T \in \T} T.$ Assume that the trees are td-maximal among trees contained in $\P$. Then, we can write 
\[
\P = \bigcup_{T \in \T^u} T \cup \bigcup_{T \in \T^l}T 
\]
where 
\begin{list}{(\arabic{listcounter})}{\usecounter{listcounter}\setcounter{listcounter}{\value{equation}}}
\item \label{firstconditionps} For each $T \in \T^u$, $T$ is a $u$-overlapping tree which is td-maximal among $u$-overlapping trees contained in $\bigcup_{T \in \T^u} T$. 
\item \label{secondconditionps} $\T^l$ is a properly-sorted collection of $l$-overlapping trees.
\item \label{thirdconditionps} $\bigcup_{T \in \T^u}T \cap \bigcup_{T \in \T^l} T = \emptyset$.
\item \label{fourthconditionps} $\{(\xi_{T},I_{T}) : T \in \T^u\} = \{(\xi_{T},I_{T}) : T \in \T^l\}  \subset \{(\xi_{T},I_{T}) : T \in \T\}$.
\end{list}\setcounter{equation}{\value{listcounter}}
\end{lemma}

\begin{proof}
After throwing out some trees, we may assume that for each $T \in \T$,
\begin{equation} \label{nonredundantcondition}
T \not\subset \bigcup_{T' \in \T \setminus \{T\}} T'.
\end{equation}
%By the td-maximality condition in the hypothesis, we then have \eqref{propsorta}.
We enumerate $\T = \{T_1, \ldots, T_N\}$ so that for each $i$, $\xi_{T_i} \leq \xi_{T_{i+1}}.$ 
Set $\P_{N+1} = \P$ and $T_{N+1} = \emptyset$ and for $i = 1, \ldots, N$ let $T_i^u$ be the maximal $u$-overlapping tree contained in $\P_{i+1}$ with top data $(\xi_{T_i},I_{T_i}),$ let $T_i^l$ be the maximal $l$-overlapping tree contained in $\P_{i+1}$ with top data $(\xi_{T_i},I_{T_i})$, and let $\P_i = \P_{i+1} \setminus (T^u_i \cup T^l_i).$ 

Set $\T^u = \{T^u_1, \ldots, T^u_N\}$ and $\T^l = \{T^l_1, \ldots, T^l_N\}.$

We now verify \eqref{propsortb}. Suppose that $P,P'' \in T^l_i$ and that $P_l < P'_l < P''_l.$  By convexity and td-maximality of the trees, we have $P,P',P'' \in T_i.$ It remains to verify that $P' \in \P_{i+1},$ but this follows immediately from the fact that $P'' \in \P_{i+1}.$

To check \eqref{propsortc} suppose that $P \in T_i^l$, $\xi_{T_j} \in \omega_{P_u},$ and $I_{T_j} \cap I_P \neq \emptyset$. Since $T_i^l$ is l-overlapping, $\xi_{T_j} > \xi_{T_i}$ and so we have $j > i.$ Since $P \in \P_{i+1}$ we know that $P \in \P_{j+1} \setminus T_j^u$; combining this with the fact that $\xi_{T_j} \in \omega_{P_u}$ implies that $I_{T_j} \subsetneq I_{P}$ contradicting \eqref{nonredundantcondition}.

Enlarging the trees in $\T^u$ to obtain td-maximality gives \eqref{firstconditionps}. Conditions \eqref{thirdconditionps}, \eqref{fourthconditionps}, and \eqref{propsorta} are clear from construction.
\end{proof}

\section{Global variation for a single tree}

The Lemma below will be used in Section \ref{poftsection} to give pointwise variation-norm estimates which are compatible with Corollary \ref{pwcorollary} and Lemmas \ref{Lpmaxnocomplemma}, \ref{Walshcrs}, and \ref{hybridlemma}. 

\begin{lemma} \label{globallemma} Suppose $1 < p < \infty$ and $r > 2$. Then for every tree $T$ which is contained in a convex tree $\overline{T}$ and which is either $l$-overlapping and $l$-convex or $u$-overlapping
\begin{equation} \label{gvstbound}
\|\sum_{P \in T} \<f,\phi_{P_l}\>\phi_{P_l}(x)1_{\omega_{P_u}}(\xi)\|_{L^p_x(V^r_\xi)} \leq C_{p,r} \size(\ov{T}) |I_{\ov{T}}|^{1/p}.
\end{equation}
\end{lemma}

\begin{proof}
Note that, by \eqref{projsubsetidentity}, the left side of \eqref{gvstbound} is equal to 
\begin{equation} \label{gvstpppbound} 
\|\sum_{P \in T} \<g,\phi_{P_l}\>\phi_{P_l}(x)1_{\omega_{P_u}}(\xi)\|_{L^p_x(V^r_\xi)}
\end{equation}
where $g = \Pi_{\ov{T}}f.$ 

First we consider the case when $T$ is $u$-overlapping. Then for each $\xi$ we have
\begin{equation} \label{gvstpppboundpw}
\sum_{P \in T} \<g,\phi_{P_l}\>\phi_{P_l}(x)1_{\omega_{P_u}}(\xi) = \sum_{\substack{P \in T \\ |I_T| \leq 2^{k_{\xi}}}} \<g,\phi_{P_l}\>\phi_{P_l}(x)
\end{equation}
where $k_{\xi}$ is the largest $k$ such that the dyadic interval of length $2^{-k}$ about $\xi_T$ contains $\xi$. Then $k_{\xi}$ is monotonic in $\xi$ on the intervals $(\xi_T,\infty)$ and $[0,\xi_T)$ and so
\eqref{gvstpppbound} is
\begin{equation} \label{gvstpppboundk}
\leq C \|\sum_{\substack{P \in T \\ |I_T| \leq 2^{k}}} \<g,\phi_{P_l}\>\phi_{P_l}(x)\|_{L^p_x(V^r_k)}.
\end{equation}
One can check that
\begin{multline*}
\sum_{\substack{P \in T \\ |I_T| \leq 2^{k}}} \<g,\phi_{P_l}\>\phi_{P_l}(x) \\
=\sum_{P \in T } \<g,\phi_{P_l}\>\phi_{P_l}(x)   - \sgn(\phi_{p_T}) D_k\left[\sgn(\phi_{p_T}) \sum_{P \in T} \<g,\phi_{P_l}\>\phi_{P_l} \right](x) 
\end{multline*}
where $p_T$ is a tile satisfying $|I_{p_T}| = |I_T|$ and $\xi_{T} \in \omega_{p_T}.$
Since $r > 2$ and $1 < p < \infty$, we may apply L\'{e}pingle's bound for the variation of martingale averages to see that the right side of \eqref{gvstpppboundk} is
\[
\leq C_{p,r} \|\sum_{P \in T } \<g,\phi_{P_l}\>\phi_{P_l}\|_{L^p} \leq C_{p,r} \|g\|_{L^p} \leq C_{p,r}\size(\ov{T}) |I_{\ov{T}}|^{1/p}
\]
as desired.

For $l$-overlapping $l$-convex $T$ one can check that, for $\xi > \xi_T$, the left side of \eqref{gvstpppboundpw}
\begin{equation} \label{lolcb1}
= \sum_{\substack{P \in T \\ |I_T| = 2^{k_{\xi}}}} \<g,\phi_{P_l}\>\phi_{P_l}(x)
\end{equation} 
where $k_{\xi}$ is the largest $k$ such that the dyadic interval of length $2^{1-k}$ about $\xi_T$ contains $\xi.$ This sum is zero when $\xi \geq \sup \omega_{P^{\min}_u(x)}$ and when $\xi < \inf \omega_{P^{\max}_u(x)}$ where $P^{\min}(x)$ and $P^{\max}(x)$ are the smallest and largest bitiles respectively from $T$ which contain $x$ in their time support. For intermediate $\xi$ it follows from $l$-convexity that there is a $P \in T$ with $|I_P| = 2^{k_\xi}$ and $x \in I_P$ and hence the right side of \eqref{lolcb1}
\[
= \sgn(\phi_{p_T}) D_{k_{\xi}}\left[\sgn(\phi_{p_T}) g \right](x).
\]
From the monotonicity of the $k_\xi$ it thus follows that \eqref{gvstpppbound} is 
\[
\leq C \|D_{k}\left[\sgn(\phi_{p_T}) g \right](x)\|_{L^p_x(V^r_k)} \leq C_{p,r} \|g\|_{L^p} \leq C_{p,r} \size(\ov{T}) |I_{\ov{T}}|^{1/p}
.
\]

\end{proof}

%The proof above also gives a pointwise inequality
%\[
%\|\sum_{P \in T} \<f,\phi_{P_l}\>\phi_{P_l}(x)1_{\omega_{P_u}}(\xi)\|_{V^r_\xi} \leq C \|D_k[\sgn(\phi_{P_T}) \Pi_S f](x)\|_{V^r_k}
%\]
%where $S = \bigcup_{P \in T}P_l$ if $T$ is $u$-overlapping and $S = \bigcup_{P \in \ov{T}}P$ if $T$ is $l$-overlapping and $l$-convex. 
%It may be possible to substitute this for the global estimate and apply a vector-valued version of L\'{e}pingle's inequality to simplify the argument in the following section.

\section{Pointwise variation for stacks of trees}

Given a collection of trees $\T$ satisfying certain assumptions, the following lemma allows us to partition $\rea^+$ into a collection of intervals $\{\upsilon_T\}_{T \in \T}$ such that the restriction of a function of the form 
\begin{equation} \label{functionoftype}
\sum_{P \in \bigcup_{T \in \T}T} c_P 1_{\omega_{P_u}}(\cdot)
\end{equation}
to the interval $\upsilon_T$ is $\sum_{P \in T} c_P 1_{\omega_{P_u}}(\cdot).$
This partitioning will be used in the current section to obtain variation-norm estimates for functions of the form \eqref{functionoftype} and it will be used in Section \ref{tssection} to obtain estimates involving the Walsh-multiplier operators induced by functions of the form \eqref{functionoftype}.

\begin{lemma} \label{pointwiselemmau}
Suppose that $x \in \rea^+$ and that $\P$ is a finite collection of bitiles with $x \in I_P$ for each $P \in \P$ and with $\P = \bigcup_{T \in \T}T$ where $\T$ satisfies one of the following two conditions:
\begin{itemize}
\item $\T$ is a collection of $u$-overlapping trees which are td-maximal among $u$-overlapping trees contained in $\P$.
\item $\T$ is a collection of properly-sorted $l$-overlapping trees. 
\end{itemize}
Then there is a collection $\{\upsilon_T\}_{T \in \T}$ of pairwise disjoint intervals covering $\rea^+$ such that for each $T \in \T$ and $\xi \in \upsilon_T$
\begin{equation} \label{intervalfortree}
\{P \in \P : \xi \in \omega_{P_u}\} \subset T.
\end{equation}
\end{lemma}

\begin{proof}
We start by proving the lemma under assumption of the first condition.
Without loss of generality assume that $|\T| \geq 2$, and that for every $T \in \T$ 
\[
T \not\subset \bigcup_{T' \in \T \setminus \{T\}} T'.
\] 
By td-maximality we see that $\xi_T \neq \xi_{T'}$ for $T \neq T'.$ Enumerate the trees $T_1, \ldots, T_N$ so that $\xi_{T_{i}} < \xi_{T_{i+1}}$ for $i=1, \ldots, N-1$. 

For each $i$ let $P^i$ be the minimal bitile in $T_i \setminus \bigcup_{i' < i} T_{i'}$, and let $\xi_i^- := \inf(\omega_{P^i_u}).$ 
We claim that $\xi_i^- > \xi_{i'}^-$ whenever $i > i'$. To see this, first note that we may assume that $\omega_{P^i_u} \cap \omega_{P^{i'}_u} \neq \emptyset$ or else the conclusion would follow from the fact that $\xi_{T_i} > \xi_{T_{i'}}.$ Since $x \in I_{P^{i'}} \cap I_{P^{i}}$ we then have $P^{i'}_u \cap P^i_u \neq \emptyset$. Thus, we must have $P^{i}_u > P^{i'}_u$ since $P^{i}_u \leq P^{i'}_u$ would imply that $P^{i} \in T_{i'}$ contradicting the definiton of $P^{i}.$ Then $\omega_{P^i} \subsetneq \omega_{P^{i'}_u}$ and so $\xi_{i}^- > \xi_{i'}^-$ as desired.
 
Set 
$\upsilon_{T_1} = [0,\xi_{2}^-), 
\upsilon_{T_N} = [\xi_{N}^-,\infty),$ 
and $\upsilon_{T_i} = [\xi_{i}^-,\xi_{i+1}^{-})$
for $1 < i < N.$ From the previous paragraph, we see that $\upsilon_{T_1}, \ldots, \upsilon_{T_N}$ are disjoint and cover $\rea^+.$ 

It remains to check that if $\xi \in \upsilon_{T_i}$ and $P \in \P$ with $\xi \in \omega_{P_u}$ then $P \in T_i.$ First assume $1 < i < N.$
Choose the minimal $i'$ such that $P \in T_{i'}$. First suppose $i' < i.$ Since $\xi \in \omega_{P_u}$,\ $\xi \geq \inf(\omega_{P^i_u}),$\ and $\xi_{T_{i'}} < \xi_{T_i},$ we must have $\omega_{P_u} \cap \omega_{P^i_u} \neq \emptyset.$ The fact that $P^i \notin T_{i'}$ rules out the possibility that $P^i_u \leq P_u$ and so we must have $P^i_u > P_u.$ But then $P \in T_i$ as desired. Now suppose $i' > i.$ By minimality of $i'$ we have $P \in T_{i'} \setminus \bigcup_{i'' < i'}T_{i''}.$ Then, by minimality of $P^{i'}$ we have $P \geq P^{i'}$ and so $\inf(\omega_{P_u}) \geq \xi_{i'}^-$ contradicting the fact that $\xi \in \omega_{P_u}.$ The appropriate halves of this argument work when $i=1$ or $N$.  

Working instead under the second condition,
for each $T \in \T$ we let 
\[
\tilde{J}_T = \bigcup_{\substack{P \in T \\ x \in I_P}} \omega_{P_u}
\]
By \eqref{propsortb} each $\tilde{J}_T$ is an interval. From \eqref{propsortc} we know that the tiles $P_u$ with $P \in \bigcup_{T \in \T}T$ are disjoint, which gives \eqref{intervalfortree} for $\xi \in \tilde{J}_T$. Combining \eqref{propsortc} with \eqref{propsorta} one sees that the intervals $\tilde{J}_T$ are pairwise disjoint. Finally, the left and right sides of \eqref{intervalfortree} are both zero for $\xi$ outside of $\bigcup_{T \in \T} \tilde{J}_T$; thus by choosing $\{J_T\}_{T \in \T}$ to be any collection of pairwise disjoint intervals which cover $\rea^+$ and which satisfy $\tilde{J}_T \subset J_T$, we are finished.    

\end{proof}

The following corollary, which can be used to obtain \eqref{mainthm}, follows immediately from the lemma above.

\begin{corollary} \label{pwcorollary} 
Suppose that $\P, \T$ and $x$ satisfy the hypotheses of Lemma \ref{pointwiselemmau}. Then for any collection of coefficients $\{c_P\}_{P \in \P} \subset \rea$
\begin{equation*} 
\|\sum_{P \in \P} c_P 1_{\omega_{P_u}}\|_{V^r} \leq C |\T|^{\frac{1}{r}} \sup_{T \in \T} \|\sum_{P \in T} c_P 1_{\omega_{P_u}}\|_{V^r}
\end{equation*}
\end{corollary}

To prove Theorem \ref{maxtruncvarmodtheorem} we will need the estimate below.

\begin{corollary} \label{pwcorollary2} 
Suppose that $\P, \T$ and $x$ satisfy the hypotheses of Lemma \ref{pointwiselemmau}. Then for any collection of coefficients $\{c_P\}_{P \in \P} \subset \rea$
\begin{equation*} 
\|\sum_{\substack{P \in \P \\ |I_P| < 2^k}} c_P 1_{\omega_{P_u}}(\xi)\|_{\ell^\infty_k(V^r_{\xi})} \leq C |\T|^{\frac{1}{r}} \sup_{T \in \T} \|\sum_{P \in T} c_P 1_{\omega_{P_u}}\|_{V^r}
\end{equation*}
\end{corollary}

\begin{proof}
Thanks to the covering in Lemma \ref{pointwiselemmau}, it suffices to observe that for each $k$ 
\begin{equation} \label{trunceqfortree}
\|\sum_{\substack{P \in T \\ |I_P| < 2^k}} c_P 1_{\omega_{P_u}}(\xi)\|_{V^r_{\xi}} \leq C
\|\sum_{P \in T} c_P 1_{\omega_{P_u}}(\xi)\|_{V^r_{\xi}}. 
\end{equation}
Let $\omega_k$ denote the dyadic interval of length $2^{-k}$ containing $\xi_T.$ First treating the case where $T$ is $u$-overlapping, we note that for $\xi \notin \omega_k$ we have 
\begin{equation} \label{kredundant}
\sum_{\substack{P \in T \\ |I_P| < 2^k}} c_P 1_{\omega_{P_u}}(\xi) = \sum_{P \in T} c_P 1_{\omega_{P_u}}(\xi)
\end{equation}
and for $\xi \in \omega_k$ we have
\[
\sum_{\substack{P \in T \\ |I_P| < 2^k}} c_P 1_{\omega_{P_u}}(\xi) = \sum_{P \in T} c_P 1_{\omega_{P_u}}(\xi')
\]
where $\xi'$ is any point in $\omega_{k-1} \setminus \omega_k$. Combining these two facts immediately implies \eqref{trunceqfortree}. 

If $T$ is instead $l$-overlapping then for $\xi \notin \omega_{k-1}$ we have \eqref{kredundant}
and for $\xi \in \omega_{k-1}$
\begin{equation} \label{kzero}
\sum_{\substack{P \in T \\ |I_P| < 2^k}} c_P 1_{\omega_{P_u}}(\xi) = 0
\end{equation}
and hence \eqref{trunceqfortree}.
\end{proof}

Finally, Theorem \ref{maxmodvartrunctheorem} is obtained from the following bound.  

\begin{corollary} \label{pwcorollary3} 
Suppose that $\P, \T$ and $x$ satisfy the hypotheses of Lemma \ref{pointwiselemmau}. Then for any collection of coefficients $\{c_P\}_{P \in \P} \subset \rea$
\begin{equation*} 
\|\sum_{\substack{P \in \P \\ |I_P| < 2^k}} c_P 1_{\omega_{P_u}}(\xi)\|_{L^\infty_{\xi}(V^r_{k})} \leq C  \sup_{T \in \T} \|\sum_{P \in T} c_P 1_{\omega_{P_u}}\|_{V^r}
\end{equation*}
\end{corollary}

\begin{proof}
Again using the covering in Lemma \ref{pointwiselemmau}, it suffices to observe that for each $\xi$ and $T$
\begin{equation} \label{trunceqfortree2}
\|\sum_{\substack{P \in T \\ |I_P| < 2^k}} c_P 1_{\omega_{P_u}}(\xi)\|_{V^r_{k}} \leq C
\|\sum_{P \in T} c_P 1_{\omega_{P_u}}(\xi')\|_{V^r_{\xi'}}. 
\end{equation}
First suppose that $T$ is $u$-overlapping and let $k_{\xi} = \sup \{k' : \xi \in \omega_{k'} \}$ (here $\omega_k$ is as defined in Corollary \ref{pwcorollary2}). Then for $k \leq k_\xi + 1$
\[
\sum_{\substack{P \in T \\ |I_P| < 2^k}} c_P 1_{\omega_{P_u}}(\xi) = \sum_{\substack{P \in T \\ |I_P| < 2^k}} c_P 
\]
and for $k > k_{\xi} + 1$ 
\[
\sum_{\substack{P \in T \\ |I_P| < 2^k}} c_P 1_{\omega_{P_u}}(\xi) = \sum_{\substack{P \in T \\ |I_P| \leq 2^{k_\xi}}} c_P. 
\]
Thus
\[
\|\sum_{\substack{P \in T \\ |I_P| < 2^k}} c_P 1_{\omega_{P_u}}(\xi)\|_{V^r_{k}} \leq 
\|\sum_{\substack{P \in T \\ |I_P| < 2^k}} c_P \|_{V^r_{k}}. 
\]
For each integer $k$ let $\xi_k$ be the left endpoint of the largest dyadic interval $\omega$ such that $\xi_T$ is in the right half of $\omega$ and $|\omega| \leq 2^{-k+1}$ (if no such dyadic interval exists, let $\xi_k = \xi_T$).
Then the points $\xi_k$ are monotonic in $k$ and
\[
\sum_{\substack{P \in T \\ |I_P| < 2^k}} c_P = \sum_{P \in T } c_P 1_{\omega_{P_u}}(\xi_k)
\]
thus
\begin{equation} \label{ktoxi}
\|\sum_{\substack{P \in T \\ |I_P| < 2^k}} c_P \|_{V^r_{k}} \leq \|\sum_{P \in T} c_P 1_{\omega_{P_u}}(\xi')\|_{V^r_{\xi'}}
\end{equation}
as desired.

For l-overlapping $T$ we have \eqref{kredundant} if $k > \max\{k' : \xi \in \omega_{k'-1}\}$ and \eqref{kzero} otherwise; hence we obtain \eqref{trunceqfortree2}.
\end{proof}

\section{Pointwise maximal multiplier estimates for stacks of trees} \label{tssection}

Suppose that $\P,\T,x,$ and $\{c_P\}_{P \in \P}$ are as in the hypotheses of Lemma \ref{pointwiselemmau}. It follows from Lemmas \ref{pointwiselemmau} and \ref{Walshcrs} that if 
$1 < q < \infty$ and $|\frac{1}{q} - \frac{1}{2}| < \frac{1}{r}$ then 
\[
\|\sum_{P \in \P} c_P 1_{\omega_{P_u}}\|_{M^q} \leq C_{q,r,\epsilon} |\T|^{|\frac{1}{q} - \frac{1}{2}| + \epsilon} \sup_{T \in \T} \|\sum_{P \in T} c_P 1_{\omega_{P_u}}\|_{V^r}.
\]
The aim of the present section is to extend this $M^q$ bound to an $M^{q,*}$ bound through the use of Lemma \ref{hybridlemma}.

\begin{lemma} \label{truncsumlemma}
Suppose that $\P,\T, x,$ and $\{c_P\}_{P \in \P}$ are as above and that $1 < q \leq 2,$ $2 < r < 2q,$ and $\epsilon > 0$. Then 
\[
\|\sum_{\substack{P \in \P \\ |I_P| < 2^k}} c_P 1_{\omega_{P_u}}\|_{M^{q,*}} \leq C_{q,r,\epsilon}  |\T|^{\frac{1}{q} - \frac{1}{r} + \epsilon} \sup_{T \in \T} \|\sum_{P \in T} c_P 1_{\omega_{P_u}}\|_{V^r}.
\]
\end{lemma}

\begin{proof}
We start by assuming the first condition in Lemma \ref{pointwiselemmau}, i.e. that 
 $\P = \bigcup_{T \in \T}T$ where $\T$ is a collection of $u$-overlapping trees which are td-maximal among all trees contained in $\P$. Without loss of generality we may also assume that $x \in I_T$ for each $T \in \T.$ Let $\Xi = \{\xi_T : T \in \T\}.$ 

If $P \in T$ and $|I_P| \geq 2^k$ then we have $\omega_{P_u}$ contained in the dyadic interval of length $2^{-k}$ about $\xi_T.$ This implies that for 
\[
\xi \in \rea^+ \setminus \bigcup_{\substack{|\omega| = 2^{-k} \\ \omega \cap \Xi \neq \emptyset}}\omega
\]
 we have
\begin{equation} \label{truncsumeq1}
\sum_{\substack{P \in \P \\ |I_P| < 2^k}} c_P 1_{\omega_{P_u}}(\xi) = 
\sum_{P \in \P } c_P 1_{\omega_{P_u}}(\xi)
. 
\end{equation}
If $\omega$ is any dyadic interval of length $2^{-k}$ and $\xi \in \omega$ then
\begin{equation} \label{truncsumeq2}
\sum_{\substack{P \in \P \\ |I_P| < 2^k}} c_P 1_{\omega_{P_u}}(\xi) = \sum_{\substack{P \in \P \\ \omega \subsetneq \omega_{P_u}}} c_P.
\end{equation}
Combining \eqref{truncsumeq1} and \eqref{truncsumeq2} we see that
\begin{multline} \label{truncsumeqmess}
\sum_{\substack{P \in \P \\ |I_P| < 2^k}} c_P 1_{\omega_{P_u}}(\xi)  =  \\ (1 - \sum_{\substack{|\omega| = 2^{-k} \\ \omega \cap \Xi \neq \emptyset}} 1_{\omega}(\xi)) \sum_{P \in \P } c_P 1_{\omega_{P_u}}(\xi)\ \ + \sum_{\substack{|\omega| = 2^{-k} \\ \omega \cap \Xi \neq \emptyset}} 1_{\omega}(\xi) \sum_{\substack{P \in \P \\ \omega \subsetneq \omega_{P_u}}} c_P.
\end{multline}
The right side of \eqref{truncsumeqmess} is the sum of three terms each of which we will bound separately. 
For the first term, we argue as indicated in the discussion at the beginning of this section. Specifically,  
Applying Lemma \ref{pointwiselemmau}, we obtain a collection of intervals $\{\upsilon_T\}_{T \in \T}$ so that 
\begin{equation} \label{dcintotreeseq}
\sum_{P \in \P } c_P 1_{\omega_{P_u}}(\xi) = \sum_{P \in T } c_P 1_{\omega_{P_u}}(\xi)
\end{equation}
for $\xi \in \upsilon_T.$
We then apply Lemma \ref{Walshcrs} with the collection of intervals $\Upsilon = \{\upsilon_T\}_{T \in \T}$ to obtain
\begin{equation*}
\|\sum_{P \in \P } c_P 1_{\omega_{P_u}} \|_{M^q} 
 \leq C_{q,r,\epsilon} |\T|^{\frac{1}{q} - \frac{1}{r} + \epsilon}  \sup_{T \in \T} \|\sum_{P \in T } c_P 1_{\omega_{P_u}}.
\|_{V^r}  
\end{equation*}
For the second term, we note when $\Xi$ is as above and $a_{\omega} = 1$ for each dyadic interval $\omega$ 
\begin{equation} \label{writeasDeltakeq}
\sum_{\substack{|\omega| = 2^{-k} \\ \omega \cap \Xi \neq \emptyset}} 1_{\omega} \sum_{P \in \P } c_P 1_{\omega_{P_u}}  = \calD_k \sum_{P \in \P } c_P 1_{\omega_{P_u}}.
\end{equation}
Combining \eqref{writeasDeltakeq} and \eqref{dcintotreeseq}, we see that Lemma \ref{hybridlemma} gives
\begin{align*}
\|\sum_{\substack{|\omega| = 2^{-k} \\ \omega \cap \Xi \neq \emptyset}} 1_{\omega}  \sum_{P \in \P } c_P 1_{\omega_{P_u}} \|_{M^{q,*}} 
&\leq C_{q,r,\epsilon} (|\Xi| + |\{\upsilon_T\}_{T \in \T}|)^{\frac{1}{q} - \frac{1}{r} + \epsilon} \sup_{T \in \T} \|\sum_{P \in T } c_P 1_{\omega_{P_u}}
\|_{V^r} \\
&\leq C_{q,r,\epsilon} |\T|^{\frac{1}{q} - \frac{1}{r} + \epsilon} \sup_{T \in \T} \|\sum_{P \in T } c_P 1_{\omega_{P_u}}\|_{V^r}
\end{align*}
Finally, for the last term we note that with $\Xi$ as above and 
\begin{equation} \label{writeasDeltak2eq}
a_{\omega} = \sum_{\substack{P \in \P \\ \omega \subsetneq \omega_{P_u}}} c_P
\end{equation}
we have
\[
\sum_{\substack{|\omega| = 2^{-k} \\ \omega \cap \Xi \neq \emptyset}} 1_{\omega} \sum_{\substack{P \in \P \\ \omega \subsetneq \omega_{P_u}}} c_P  = \calD_k 
\]
and so applying Lemma \ref{Lpmaxnocomplemma}
\begin{equation} \label{thirdtermu}
\|\sum_{\substack{|\omega| = 2^{-k} \\ \omega \cap \Xi \neq \emptyset}} 1_{\omega} \sum_{\substack{P \in \P \\ \omega \subsetneq \omega_{P_u}}} c_P \|_{M^{q,*}} 
 \leq C_{q,r,\epsilon} |\T|^{\frac{1}{q} - \frac{1}{r} + \epsilon} \sup_{T \in \T} \|\sum_{|\omega| = 2^k} 1_{\omega}(\xi_T) \sum_{\substack{P \in \P \\ \omega \subsetneq \omega_{P_u}}} c_P \|_{V^r_k}.
\end{equation}
For each $k$, 
\begin{align*}
\sum_{|\omega| = 2^k} 1_{\omega}(\xi_T) \sum_{\substack{P \in \P \\ \omega \subsetneq \omega_{P_u}}} c_P 
= \sum_{\substack{P \in \P \\ |I_P| < 2^k}} c_P 1_{\omega_{P_u}}(\xi_T)
= \sum_{\substack{P \in T' \\ |I_P| < 2^k}} c_P 
\end{align*}
where $T' \in \T$ is the tree containing the maximal element of $\P$ satisfying $\xi_T \in \omega_{P_u}, x \in I_P$, and the last identity follows from td-maximality of the tree $T'$. Therefore, the right side of \eqref{thirdtermu} is 
\begin{align*}
&\leq C_{q,r,\epsilon} |\T|^{\frac{1}{q} - \frac{1}{r} + \epsilon} \sup_{T \in \T} \|\sum_{\substack{P \in T \\ |I_P| < 2^k}} c_P \|_{V^r_k}  \\
&\leq C_{q,r,\epsilon} |\T|^{\frac{1}{q} - \frac{1}{r} + \epsilon} \sup_{T \in \T} \|\sum_{P \in T} c_P 1_{\omega_{P_u}}(\xi)\|_{V^r_\xi}  
\end{align*}
where the second inequality follows from \eqref{ktoxi}. Thus, we obtain the proof of the lemma under the assumption of the first condition. 

We now assume the second condition, i.e. that $\T$ is a properly-sorted collection of $l$-overlapping trees. Again, assume that $x \in I_T$ for each $T \in \T.$ Consider 
\[
\sum_{\substack{P \in \P \\ |I_P| > 2^k}} c_P 1_{\omega_{P_u}}(\xi)
\] 
If $P \in T$ and $I_P > 2^k$ then $\omega_P$ is contained in the dyadic interval of length $2^{-k}$ containing $\xi_T$ and so, letting $\Xi = \{\xi_T : T \in \T\}$ we have
\[
1_{\omega_{P_u}}(\xi) = 1_{\omega_{P_u}}(\xi) \sum_{\substack{|\omega| = 2^{-k} \\ \omega \cap \Xi \neq \emptyset}} 1_{\omega}(\xi)
\] 
where we sum over dyadic intervals $\omega.$ If $P \in T$ with $|I_P| \leq 2^k$ then $\omega_{P_u}$ does not intersect the dyadic interval of length $2^{-k}$ about $\xi_T$ and, furthermore, by \eqref{propsortc} does not intersect the dyadic interval of length $2^{-k}$ about any $\xi_{T'} > \xi_T$ for  $T' \in \T.$ This gives 
\[
 1_{\omega_{P_u}}(\xi) \sum_{\substack{|\omega| = 2^{-k} \\ \omega \cap \Xi \neq \emptyset}} 1_{\omega}(\xi) = 0
\] 
and hence 
\[
\sum_{\substack{P \in \P \\ |I_P| > 2^k}} c_P 1_{\omega_{P_u}}(\xi) = 
\sum_{\substack{|\omega| = 2^{-k} \\ \omega \cap \Xi \neq \emptyset}} 1_{\omega}(\xi) 
\sum_{P \in \P} c_P 1_{\omega_{P_u}}(\xi). 
\]
Thus, 
\[
\sum_{\substack{P \in \P \\ |I_P| \leq 2^k}} c_P 1_{\omega_{P_u}}(\xi) = 
(1 - \sum_{\substack{|\omega| = 2^{-k} \\ \omega \cap \Xi \neq \emptyset}} 1_{\omega}(\xi))
\sum_{P \in \P} c_P 1_{\omega_{P_u}}(\xi)
\]
and the remaining argument follows exactly that for the first two terms in \eqref{truncsumeqmess}.
\end{proof}

\section{Proof of Theorems} \label{poftsection}
Theorem \ref{walshMpstar} is established by Using Lemma \ref{truncsumlemma} to apply the following proposition with $\|\cdot\|_{\calN} = \|\cdot\|_{M^{q,*}}$, $\eta_{P,k} = 1_{(-\infty,2^k)}(|I_P|),$ and $r$ sufficiently close to $2$. 
Using Corollary \ref{pwcorollary} to apply the proposition with $\|\cdot\|_{\calN_{\xi,k}} = \|\cdot\|_{\ell^{\infty}_k(V^r_\xi)}$ and $\eta_{P,k} = 1$ establishes \eqref{mainthm}. 
Using Corollary \ref{pwcorollary2} to apply the proposition with $\|\cdot\|_{\calN_{\xi,k}} = \|\cdot\|_{\ell^{\infty}_k(V^r_\xi)}$ and $\eta_{P,k} = 1_{(-\infty,2^k)}(|I_P|)$ establishes Theorem \ref{maxtruncvarmodtheorem}. 
Using Corollary \ref{pwcorollary3} to apply the proposition with $\|\cdot\|_{\calN_{\xi,k}} = \|\cdot\|_{L^{\infty}_\xi(V^r_k)}$ and $\eta_{P,k} = 1_{(-\infty,2^k)}(|I_P|)$ establishes Theorem \ref{maxmodvartrunctheorem}. 

\begin{proposition} \label{poftprop}
Let $r > 2$, $1 < p < \infty$, and let $\{\eta_{P,k}\}_{P \in \P_0 , k \in \BBZ} \subset \rea$ be a collection of coefficients. Suppose that for all $\P,\T, x,$ and $\{c_P\}_{P \in \P}$ as in the hypotheses of Corollary \ref{pwcorollary}, a norm $\|\cdot\|_{\calN}$ acting on functions defined on $\rea^+ \times \BBZ$ satisfies
\begin{equation} \label{poftassumption} 
\|\sum_{P \in \P} \eta_{P,k} c_P 1_{\omega_{P_u}}\|_{\calN} \leq C |\T|^{\alpha} \sup_{T \in \T} \|\sum_{P \in T} c_P 1_{\omega_{P_u}}\|_{V^r}
\end{equation}
for some $\alpha < \min(1 - \frac{1}{p},\frac{1}{2}).$ Then
\[
\|\sum_{P \in \P_0} \eta_{P,k} \<f,\phi_{P_l}\> \phi_{P_l}(x) 1_{\omega_{P_u}}(\xi)\|_{L^p_x(\calN_{\xi,k})} \leq C_{p,r} \|f\|_{L^p} 
\]
\end{proposition}

\begin{proof}
We will prove a restricted weak-type estimate; the full result follows by interpolation. Specifically, we suppose that $|f| \leq 1_F$ and $\lambda > 0$ and want to show
\[
|\{ x : \| \sum_{\substack{P \in \P_0} } \eta_{P,k}\<f,\phi_{P_l}\>\phi_{P_l}(x)1_{\omega_{P_u}}(\xi)\|_{\calN_{\xi,k}} > \lambda \}| \leq C_{p,r} |F|/\lambda^p.
\]
The inequality above will be demonstrated by covering the set on the left side by ``exceptional sets'' $E_1,E_2^u,E_2^l,E_3$ of acceptably small measure. 

We begin by treating the case $\lambda < 1$. There, we set
\[
E_1 := \{M[1_F] \geq c \lambda^p\}
\]
where $M$ is the dyadic Hardy-Littlewood maximal operator. By the weak-type 1-1 estimate for $M$ we have $|E_1| \leq C |F|/\lambda^p$. Since we only need to bound the $\calN$-norm for $x \notin E_1$, we can assume that for every $P \in \P_0$ we have $I \not\subset E_1$ and hence, since the $L^\infty$ norm of the phase-space projection onto any $T$ can be controlled by the projection onto a subtile of an element of $T$,
\[
\size(\P_0,f) \leq C \lambda^p.
\]

For each $n \geq 0$ we apply Lemma \ref{treeselection} with $\P = \P_n$ and set $\P_{n+1}:=\P'$ so that 
\[
\size(\P_{n+1},f) \leq C 2^{-(n+1)} \lambda^p
\]
For each $n$ we apply Lemma \ref{propsortlemma} to the collection of bitiles $\P_n \setminus \P_{n+1}$ and obtain collections of u-overlapping trees $\T^u_n$ and l-overlapping trees $\T^l_n.$ 

Fix $s > 0$ large and $\epsilon > 0$ small  with magnitudes to be determined later and let 
\[
\gamma_{n} = c 2^{-n} \lambda^p (2^{2n} \lambda^{-p})^{1/s} 2^{\epsilon n}.
\]  
Each $T \in \T^u_n$ is contained, by construction, in a convex tree $\ov{T}$ with $\size(\ov{T},f) \leq C 2^{-n} \lambda^p$ and $I_{\ov{T}} = I_T.$ Also, we have
\begin{equation} \label{L1topspf}
\|\sum_{T \in \T_n^u} 1_{I_T}\|_{L^1} \leq C 2^{2n} \lambda^{-2p} |F|.
\end{equation}
Thus, letting 
\[
E_2^u := \bigcup_{n \geq 0} \bigcup_{T \in \T_n^u} \{x : \|\sum_{\substack{P \in T }} \<f,\phi_{P_{l}}\>\phi_{P_{l}}(x)1_{\omega_{P_u}}(
\xi)\|_{V^r_\xi} > \gamma_n \}
\]
we apply Lemma \ref{globallemma} to obtain 
\begin{align*}
|E_2^u| &\leq \sum_{n \geq 0} \sum_{T \in \T_n^u} \gamma^{-s}_n \|\sum_{\substack{P \in T }} \<f,\phi_{P_{l}}\>\phi_{P_{l}}(x)1_{\omega_{P_u}}(
\xi)\|_{L^s_x(V^r_\xi)}^s\\
&\leq C \sum_{n \geq 0} \sum_{T \in \T^u_n} \gamma^{-s}_n  (2^{-n}\lambda^p)^s |I_T| \\
&\leq C \sum_{n \geq 0} \gamma^{-s}_n  (2^{-n} \lambda^p)^s  2^{2n} \lambda^{-2p} |F|\\
&\leq C \sum_{n \geq 0} 2^{-s\epsilon n} |F|/\lambda^p\\
&\leq C |F|/\lambda^p.
\end{align*}

Defining $E_2^l$ analogously, we obtain the same bound.

Let 
\[
\beta_n = c 2^{2n}\lambda^{-p} 2^{\epsilon n}
\]
and 
\[
E_3 := \bigcup_{n \geq 0}\{x : \sum_{T \in \T_n^u} 1_{I_T}(x) > \beta_n\}.
\]
Applying \eqref{L1topspf}, we have
\begin{align*}
|E_3| &\leq C\sum_{n \geq 0} \beta_n^{-1} 2^{2n} \lambda^{-2p} |F| \\
&\leq C \sum_{n \geq 0} 2^{-\epsilon n } |F|/\lambda^p \\
&\leq C |F|/\lambda^p.
\end{align*}
Recall that the trees in $\T^u_n$ and $\T^l_n$ have shared top data and so $E_3$ also gives control over $T \in \T^l_n$.

Fix $x \notin E_1 \cup E_2^u \cup E_2^l \cup E_3;$ we need to show that 
\begin{equation*} 
\| \sum_{\substack{P \in \P_0} } \eta_{P,k} \<f,\phi_{P_l}\>\phi_{P_l}(x)1_{\omega_{P_u}}(\xi)\|_{\calN_{\xi,k}} \leq \lambda.
\end{equation*}
Since every $P \in \P_0$ with $\<f,\phi_{P_l}\> \neq 0$ is in $\P_n \setminus \P_{n+1}$ for some $n$, the left side above is
\begin{equation} \label{sublovern}
\leq \sum_{n \geq 0} \| \sum_{\substack{P \in \P_n \setminus \P_{n+1}} } \eta_{P,k} \<f,\phi_{P_l}\>\phi_{P_l}(x)1_{\omega_{P_u}}(\xi)\|_{\calN_{\xi,k}}.
\end{equation}
For each $n$, we have 
\[
\P_n \setminus \P_{n+1} = \bigcup_{T \in \T_n^u}T \cup \bigcup_{T \in \T_n^l}T
\]
and so by \eqref{thirdconditionps} we see that the $n$'th term in \eqref{sublovern} is
\begin{multline}
\leq \| \sum_{\substack{P \in \bigcup_{T \in \T^u_n}T} } \eta_{P,k}\<f,\phi_{P_l}\>\phi_{P_l}(x)1_{\omega_{P_u}}(\xi)\|_{\calN_{\xi,k}} \\ + \| \sum_{\substack{P \in \bigcup_{T \in \T^l_n} T} } \eta_{P,k}\<f,\phi_{P_l}\>\phi_{P_l}(x)1_{\omega_{P_u}}(\xi)\|_{\calN_{\xi,k}}.
\end{multline}
Letting $\tilde{\T}^u_n = \{(T \cap \{P \in \P_0 : x \in I_P\}, \xi_T, I_T) : T \in T^u_n, x \in I_T\}$ and similarly for $\tilde{\T}^l_n$, the display above is clearly 
\begin{multline}
= \| \sum_{\substack{P \in \bigcup_{T \in \tilde{\T}^u_n}T} } \eta_{P,k}\<f,\phi_{P_l}\>\phi_{P_l}(x)1_{\omega_{P_u}}(\xi)\|_{\calN_{\xi,k}} \\ + \| \sum_{\substack{P \in \bigcup_{T \in \tilde{\T}^l_n} T} } \eta_{P,k}\<f,\phi_{P_l}\>\phi_{P_l}(x)1_{\omega_{P_u}}(\xi)\|_{\calN_{\xi,k}}.
\end{multline}
Noting that $\T^l_n$ is still properly sorted, and that the $\T^u_n$ are still td-maximal among $u$-overlapping trees contained in $\bigcup_{T \in \tilde{\T}^u_n}T$, we may apply \eqref{poftassumption} with $c_P = \<f,\phi_{P_l}\>\phi_{P_l}(x)$ to see that the display above is   
\begin{align*}
&\leq C (|\tilde{\T}^l_n| + |\tilde{\T}^u_n|)^{\alpha} \sup_{T \in \tilde{\T}^u_n \cup \tilde{\T}^l_n} \|\sum_{\substack{P \in T }} \<f,\phi_{P_{l}}\>\phi_{P_{l}}(x)1_{\omega_{P_u}}(
\xi)\|_{V^r_\xi} \\ 
&\leq C \beta_n^{\alpha} \gamma_n \\
&\leq c 2^{-n(1 - 2 \alpha)} 2^{n(\frac{2}{s} + (\alpha + 1)\epsilon)} \lambda^{p(1 - \alpha)} \lambda^{-\frac{p}{s}}
\end{align*}
Since $\alpha < \min(\frac{1}{2}, 1 - \frac{1}{p})$ we may choose $\epsilon$ sufficiently small and $s$ sufficiently large so that the right side above is $\leq c 2^{-\tilde{\epsilon}n} \lambda$
for some $\tilde{\epsilon} > 0$ and hence summing over $n$ and choosing $c$ sufficiently small, we obtain the desired bound for $\lambda < 1$.

In the case that $\lambda \geq 1$ we set $E_1 = \emptyset$ and use the bound $\size(\P_0,f) \leq C$. We decompose $\P_0$ as in the case $\lambda < 1$ so that $\size(\P_n,f) \leq C 2^{-n}.$ Letting
\[
\gamma_n = c 2^{-n} (2^{2n} \lambda^{p})^{1/s} 2^{\epsilon n} 
\] 
we define $E_2^u,E_2^l$ as above and obtain $|E_2^u|,|E_2^l| \leq C |F|/\lambda^p$

Interpolating the bounds 
\begin{equation*} 
\|\sum_{T \in \T_n^u} 1_{I_T}\|_{L^1} \leq C 2^{2n} |F|
\end{equation*}
and 
\begin{equation*} 
\|\sum_{T \in \T_n^u} 1_{I_T}\|_{BMO} \leq C 2^{2n}
\end{equation*}
we see that
\begin{equation*} 
\|\sum_{T \in \T_n^u} 1_{I_T}\|_{L^t} \leq C 2^{2n} |F|^{1/t}
\end{equation*}
where $t < \infty$ is some fixed exponent which will be chosen sufficiently large in a manner to be determined.
Then for each $\beta$
\[
|\{\sum_{T \in \T_n^u} 1_{I_T} > \beta\}| \leq C \beta^{-t} 2^{2tn} |F| 
\]
so, letting
\[
\beta_n = c 2^{2n}\lambda^{\frac{p}{t}} 2^{\epsilon n} 
\]
we define $E_3$ as above and have $|E_3| \leq C |F|/\lambda^p.$

For $x \notin E_2^u \cup E_2^l \cup E_3$ we thus have
\begin{align*}
\| \sum_{\substack{P \in \P_n \setminus \P_{n+1}} } \eta_{P,k}\<f,\phi_{P_l}\>\phi_{P_l}(x)1_{\omega_{P_u}}(\xi)\|_{\calN_{\xi,k}} &\leq C \beta_n^{\alpha} \gamma_n \\
&\leq c 2^{-n(1 - 2 \alpha)} 2^{n(\frac{2}{s} + (\alpha + 1)\epsilon)} \lambda^{p(\frac{\alpha}{t} + \frac{1}{s})}.\\
\end{align*}
Summing over $n$, this is $\leq \lambda$ provided that $s,t$ are chosen sufficiently large and $\epsilon,c$ are chosen suffiently small.
\end{proof}

\section{Variation-norm estimates for multipliers} \label{varmulsection}

The following is an $s$-variation-norm analog of Lemma \ref{truncsumlemma}. By taking $r$ sufficiently (depending on $p,q,s$) close to $2$ it implies Theorem \ref{walshMpvar} through the use of Proposition \ref{poftprop}.

\begin{lemma} \label{truncsumvarlemma}
Suppose that $\P,\T, x,$ and $\{c_P\}_{P \in \P}$ are as in Lemma \ref{truncsumlemma} and that $1 < q \leq 2$, $2 < r < 2q$, $\epsilon > 0$ and $r < s$. Then 
\[
\|\sum_{\substack{P \in \P \\ |I_P| < 2^k}} c_P 1_{\omega_{P_u}}\|_{M^{q,s}} \leq C_{q,r,s,\epsilon} |\T|^{(\frac{1}{2} - \frac{1}{r})\frac{s}{s-2} + \frac{1}{q} - \frac{1}{2} + \epsilon} \sup_{T \in \T} \|\sum_{P \in T} c_P 1_{\omega_{P_u}}\|_{V^r}.
\]
\end{lemma}

Except for Lemma \ref{varbourgainlemma} (which is a key element in the proof of Lemma \ref{hybridintervallemma}), each step in the proof of Lemma \ref{truncsumlemma} is insensitive to the difference between the $M^{q,s}$ and the $M^{q,*}$ norms. Thus, to establish Lemma \ref{truncsumvarlemma} it suffices to prove the following variation-norm extension of Lemma \ref{varbourgainlemma}.

\begin{lemma} \label{Lpvarnocomplemma}
Let $r > 2, \epsilon > 0,$ and $\Xi \subset \rea^+.$ Then 
\[
\|\calD_k\|_{M^{2,s}} \leq C_{r,s,\epsilon} |\Xi|^{(\frac{1}{2} - \frac{1}{r})\frac{s}{s-2} + \epsilon} \sup_{\xi \in \Xi} \|\sum_{|\omega| = 2^k } a_{\omega}1_{\omega}(\xi)\|_{V^r_k}.
\]
\end{lemma}

To prove Lemma \ref{Lpvarnocomplemma}, one follows the method used to prove Lemma \ref{varbourgainlemma} in \cite{demeter08twm} with some refinements which we will now elaborate on. The main advance needed is the following variation-norm version of Proposition 4.2 from \cite{demeter08twm}. 

\begin{proposition} \label{varfuncprop}
Let $H$ be a Hilbert space, $A$ be a finite measure space, $2 < r < s$, and $\delta > 0$. Suppose that we are given a function $g$ from $A$ to $H$ such that for each $a \in A,$ $|g(a)| \leq \delta$ and such that for each $h \in H$
\begin{equation} \label{almostorthocond}
\|\<g(a),h\>\|_{L^2(A)} \leq |h|.
\end{equation}
Then, for each sequence $\{c_k\}_{k \in \BBZ}$ of points in $H$ 
\begin{equation} \label{varfuncbound}
\|\<g(a),c_k\>\|_{L^2_a(V^s_k)} \leq C_{r,s} (\delta^2|A|)^{(\frac{1}{2} - \frac{1}{r})\frac{s}{s-2}}  \|c_k\|_{V^r}.
\end{equation}
\end{proposition}

The proof of Proposition \ref{varfuncprop} uses the same method as that of Lemma 3.2 in \cite{nazarov10czd}. However, since the statement is more general here, we will repeat the argument. 

\begin{proof}
Let 
\[
\|\cdot\|_{\tilde{V}_s} = \|\cdot\|_{V_s} - \|\cdot\|_{L^{\infty}}.
\]
By Proposition 4.2 of \cite{demeter08twm}, it suffices to prove \eqref{varfuncbound} with the $\tilde{V}^s$ norm in place of the $V^s$ norm. By a limiting argument, we may also assume that our sequence $\{c_k\}_{k=1}^M$ has finite length, provided that $C_{r,s}$ is independent of $M$. 

For each $\lambda > 0$ we cover $\{c_k\}_{k=1}^M$ with respect to $\lambda$-jumps as follows. Set $l(\lambda,1) = 1.$ Suppose that $l(\lambda,1) < \ldots <  l(\lambda,L)$ have been chosen, and let $B(c_{l(\lambda,L)},\lambda)$ denote the ball of radius $\lambda$ centered at $c_{l(\lambda,L)}.$ If $\{c_k : k > l(\lambda,L)\} \subset B(c_{l(\lambda,L)},\lambda)$ then stop and set $L_\lambda = L$ and $l(\lambda,L+1) = \infty.$ Otherwise, let $l(\lambda,L+1)$ be chosen minimally with $l(\lambda,L+1) > l(\lambda,L)$ and $c_{l(\lambda,L+1)} \notin B(c_{l(\lambda,L)},\lambda).$ This process will stop, yielding some $L_{\lambda} \leq M$. It is clear that 
\begin{equation} \label{mefromvar}
\lambda (L_{\lambda} - 1)^{1/r} \leq \|c_k\|_{V^r}\ \ .
\end{equation}

We now define a recursive ``parent'' function based on the covering above. 
Fix some $\lambda_0 < \min\{|c - c'| : c,c' \in \{c_k\}_{k=1}^M \text{\ and } c \neq c' \}.$
For $k=1, \ldots, M$ define $\rho(-1,k) = k$. Once $\rho(n,k)$ has been defined for $n=-1, \ldots, L$ set $\rho(L+1,k) = l(2^{L+1}\lambda_0,m)$ where $m$ is the unique integer satisfying
\[
l(2^{L+1}\lambda_0,m) \leq \rho(L,k) < l(2^{L+1}\lambda_0,m+1).
\]
Notice that we have
\[
|c_{\rho(n,k)} - c_{\rho(n+1,k)}| < 2^{n+1}\lambda_0
\]
and in particular $c_{\rho(0,k)} = c_k.$ Also note that $\rho(n,k) = 1$  whenever $2^n\lambda_0 \geq \diam(\{c_k\}_{k=1}^M).$ Thus
\[
c_k = c_1 + \sum_{n=0}^{\infty} c_{\rho(n,k)} - c_{\rho(n+1,k)}\ \ .
\]
Finally, by induction, one sees that $\rho(n,k)$ is nondecreasing in $k$ for each fixed $n$.

We have
\[
\|\<g(a), c_{k}\> \|_{L^2_a(\tilde{V}^s_k)}
 \leq \sum_{n = 0}^\infty \|\<g(a),(c_{\rho(n,k)} - c_{\rho(n+1,k)})\>\|_{L^2_a(\tilde{V}^s_k)}\ \ .
\]
Observe that the right hand side above
\[
= \sum_{n : L_{2^n\lambda_0} > 1} \|\<g(a), c_{\rho(n,k)} - c_{\rho(n+1,k)}) \>\|_{L^2_a(\tilde{V}^s_k)}.
\]
Using the monotonicity of the $\rho(n,\cdot)$ 
 and the fact that the range of $\rho(n,\cdot)$ is contained in $\{l(2^n\lambda_0,m) : m = 1, \ldots, L_{2^n\lambda_0}\}$ we see that the display above is 
\[
\leq 2 \sum_{n: L_{2^n\lambda_0} > 1}\| \left(\sum_{m = 1}^{L_{2^n\lambda_0}} |\<g(a),(c_{l(2^n\lambda_0,m)} - c_{\tilde{\rho}(n+1,l(2^n\lambda_0,m))})\>|^s \right)^{1/s} \|_{L^2_a}
\]
where we let $\tilde{\rho}(n+1,l(2^n\lambda_0,m))$ denote $l(2^{n+1}\lambda_0,i)$ where $i$ is the unique integer satisfying
\[
l(2^{n+1}\lambda_0,i) \leq l(2^n\lambda_0,m) < l(2^{n+1}\lambda_0,i + 1).
\]
Estimating $\ell^s$ by $\ell^2$, switching the order of integration, and using \eqref{almostorthocond}, we see that the $n$'th term in the outer sum above is
\[
\leq C 2^{n}\lambda_0 L_{2^n\lambda_0}^{1/2} 
\leq C (2^n\lambda_0)^{1 - \frac{r}{2}} \|c_k\|_{V^r}^{\frac{r}{2}}.
\]

We can also estimate the $n$'th term by
\begin{multline*}
\| \left(\sum_{m = 1}^{L_{2^n\lambda_0}} (\delta|c_{l(2^n\lambda_0,m)} - c_{\tilde{\rho}(n+1,(2^n\lambda_0,m))}|)^s \right)^{1/s} \|_{L^2_a} \\\leq \delta |A|^{1/2} \left(\sum_{m = 1}^{L_{2^n\lambda_0}} |c_{l(2^n\lambda_0,m)} - c_{\tilde{\rho}(n+1,l(2^n\lambda_0,m))}|^s \right)^{1/s} 
\\ \leq \delta |A|^{1/2} (2^n\lambda_0)^{1 - \frac{r}{s}} \|c_k\|_{V^r}^{\frac{r}{s}}.
\end{multline*}
Choosing whichever of the two bounds is favorable for each $n$ and summing gives the desired result.

\end{proof}

Through the averaging argument in the proof of Corollary 4.3 from \cite{demeter08twm}, one sees that Proposition \ref{varfuncprop} implies the following two corollaries.

\begin{corollary} \label{wellsepvarbourgain}
Let $r > 2, 1 < q \leq 2,$ and $\Xi \subset \rea^+.$ If no two elements of $\Xi$ are contained in the same dyadic interval of length $1$ then 
\[
\|1_{(-\infty,0]}(k)\calD_k\|_{M^{2,s}} \leq C_{r,s} |\Xi|^{(\frac{1}{2} - \frac{1}{r})\frac{s}{s-2}} \sup_{\xi \in \Xi} \|\sum_{|\omega| = 2^k } a_{\omega}1_{\omega}(\xi)\|_{V^r_k}.
\]
\end{corollary}

\begin{corollary} \label{wellsepvarbourgain2}
Let $r > 2, 1 < q \leq 2,$ and $\Xi \subset \rea^+.$ Suppose that for each dyadic interval $\omega$ of length $1$ and each $k \in \BBZ$ we have a coefficient $a_{\omega,k} \in \rea.$ Then  
\[
\|\sum_{\substack{|\omega| = 1 \\ \omega \cap \Xi \neq \emptyset}} a_{\omega,k} 1_{\omega}\|_{M^{2,s}} \leq C_{r,s} |\Xi|^{(\frac{1}{2} - \frac{1}{r})\frac{s}{s-2}} \sup_{\xi \in \Xi} \|\sum_{|\omega| = 1 } a_{\omega,k} 1_{\omega}(\xi)\|_{V^r_k}.
\]
\end{corollary}

Finally, to see that Corollaries \ref{wellsepvarbourgain} and \ref{wellsepvarbourgain2} imply Lemma \ref{Lpvarnocomplemma}, one argues almost (the substitution of the Walsh-Paley transform for the Fourier transform allows minor technical simplifications) exactly as in the proof of Theorem 4.3 of \cite{nazarov10czd}.

\bibliographystyle{hamsplain}
\bibliography{roberlin}
\end{document}